\documentclass[11pt]{amsart}
\usepackage{amsmath,amssymb}
\numberwithin{equation}{section}
\usepackage{xcolor}

\newtheorem{theorem}{Theorem}[section]
\newtheorem{thm}[theorem]{Theorem}
\newtheorem{lemma}[theorem]{Lemma}
\newtheorem{lem}[theorem]{Lemma}
\newtheorem{prop}[theorem]{Proposition}
\newtheorem{coro}[theorem]{Corollary}
\newtheorem{conj}[theorem]{Conjecture}

\theoremstyle{definition}
\newtheorem{definition}[theorem]{Definition}
\newtheorem{defi}[theorem]{Definition}

\newtheorem{exa}[theorem]{Example}
\newtheorem{quest}[theorem]{Question}

 \newtheorem*{ackn}{Acknowledgements}
 
 \newtheorem*{thmA}{Theorem A} 
 \newtheorem*{thmB}{Theorem B} 
\newtheorem*{thmC}{Theorem C}

 \newcommand{\R}{\mathbb R}
 
 \newcommand{\C}{\mathbb C}
  
 \newcommand{\N}{\mathbb N}

 \newcommand{\e}{\varepsilon}

 \newcommand{\f}{\varphi}
 
 \newcommand{\p}{\psi}

 \newcommand \psh {{\rm PSH}}
 \newcommand \PSH {{\rm PSH}}

 \usepackage{hyperref}
\hypersetup{
    unicode=false,        
    pdftoolbar=true,      
    pdfmenubar=true,       
    pdffitwindow=false,     
    pdfstartview={FitH},    
    pdftitle={Quasi-psh envelopes 2},    
    pdfauthor={Guedj, Lu},     
    colorlinks=true,       
   linkcolor=blue,          
    citecolor=blue,        
    filecolor=blue,      
    urlcolor=blue}



\begin{document}

\title[Quasi-plurisubharmonic envelopes 2]{Quasi-plurisubharmonic envelopes 2: Bounds on Monge-Amp\`ere volumes}

\author{Vincent Guedj \& Chinh H. Lu}

\address{Institut de Math\'ematiques de Toulouse   \\ Universit\'e de Toulouse \\
118 route de Narbonne \\
31400 Toulouse, France\\}

\email{\href{mailto:vincent.guedj@math.univ-toulouse.fr}{vincent.guedj@math.univ-toulouse.fr}}
\urladdr{\href{https://www.math.univ-toulouse.fr/~guedj}{https://www.math.univ-toulouse.fr/~guedj/}}

\address{Universit\'e Paris-Saclay, CNRS, Laboratoire de Math\'ematiques d'Orsay, 91405, Orsay, France.\\
Current address: Univ Angers, CNRS, LAREMA, SFR MATHSTIC, F-49000 Angers, France.}


\email{\href{mailto:hoangchinh.lu@univ-angers.fr}{hoangchinh.lu@univ-angers.fr}}
\urladdr{\href{https://math.univ-angers.fr/~lu/}{https://math.univ-angers.fr/~lu/}}

\date{\today}


 \begin{abstract}
  In \cite{GL21a} we have developed   a new   approach to $L^{\infty}$-a priori estimates
  for degenerate complex Monge-Amp\`ere equations,
when the reference form is  closed. 
This simplifying assumption was used to ensure the constancy of the  volumes of Monge-Amp\`ere  measures.

We study here the way these volumes stay away from zero and infinity
when the reference form is no longer closed.
We establish a transcendental version of the Grauert-Riemenschneider conjecture,
partially answering conjectures of Demailly-P\u{a}un \cite{DP04} and Boucksom-Demailly-P\u{a}un-Peternell \cite{BDPP13}.

Our approach relies on a fine use of quasi-plurisubharmonic envelopes.
The results obtained here will be used in  \cite{GL21b} for solving
degenerate complex Monge-Amp\`ere equations on compact Hermitian varieties.
\end{abstract}


 \maketitle

\tableofcontents


\section*{Introduction}

The study of complex Monge-Amp\`ere equations on compact Hermitian (non K\"ahler) manifolds
has gained considerable interest in the last decade, after Tosatti and Weinkove
established an appropriate version of Yau's theorem in \cite{TW10}.
The smooth Gauduchon-Calabi-Yau conjecture has been further solved   by
Sz\'ekelyhidi-Tosatti-Weinkove \cite{STW17}, while the pluripotential
theory has been partially extended by Dinew, Ko{\l}odziej, and Nguyen \cite{DK12,KN15,Din16, KN19}.

As in Yau's original proof \cite{Yau78}, the method of \cite{TW10} consists in establishing a priori estimates along a continuity path,
and the most delicate estimate turns out again to be the a priori $L^{\infty}$-estimate.
The fact that the reference form is not closed introduces several new difficulties:
there are many extra terms to handle when using Stokes theorem,
and it becomes non trivial to get uniform bounds on the total Monge-Amp\`ere volumes
involved in the estimates.

In \cite{GL21a} we have developed 
a new approach for establishing   uniform a priori estimates,
restricting to the context of K\"ahler manifolds for simplicity.
While the pluripotential approach consists in measuring the Monge-Amp\`ere capacity of
sublevel sets $(\f<-t)$, we directly measure the volume of the latter, avoiding
delicate integration by parts. Our approach   applies in
  the Hermitian  setting, once certain Monge-Amp\`ere 
  volumes are under control.
  Understanding the behavior of these volumes is the main focus of this article,
  while \cite{GL21b} is concerned with the resolution of degenerate complex Monge-Amp\`ere equations.

 \smallskip

We let $X$ denote a compact complex manifold of complex dimension $n$, equipped with a 
Hermitian metric $\omega_X$. The first difficulty we face is to decide whether
$$
v_+(\omega_X):=\sup \left\{ \int_X (\omega_X+dd^c \f)^n \; : \; \f \in\PSH(X,\omega_X) \cap L^{\infty}(X) \right\}
$$
is finite.
Here $d=\partial+\overline{\partial}$, $d^c=\frac{1}{2i} (\partial-\overline{\partial})$,
and $\PSH(X,\omega_X)$ is the set of $\omega_X$-plurisubharmonic functions: these are
functions $u:X \rightarrow \R \cup \{-\infty\}$ which are locally given as the sum of 
a smooth and a plurisubharmonic function, and such that
$\omega_X+dd^c u \geq 0$ is a positive current.
The complex Monge-Amp\`ere measure $(\omega_X+dd^c u)^n$ is well-defined by \cite{BT82}.

Building on works of Chiose \cite{Chi16} and Guan-Li \cite{GL10}
we provide several   results which ensure that the condition $v_+(\omega_X)<+\infty$
 is  satisfied:
\begin{itemize}
\item for any compact complex manifold $X$ of dimension $n \leq 2$;
\item for any threefold which admits a pluriclosed metric $dd^c \tilde{\omega}_X=0$;
\item as soon as there exists a metric $\tilde{\omega}_X$ such that $dd^c \tilde{\omega}_X=0$ and $dd^c \tilde{\omega}_X^2=0$;
\item as soon as $X$ belongs to the Fujiki class ${\mathcal C}$.
\end{itemize}
The Fujiki class   is the class of compact complex manifolds that are bimeromorphically 
 equivalent to K\"ahler manifolds.
 
 \smallskip

We also need to bound the Monge-Amp\`ere volumes from below.
Given a semi-positive form $\omega$, we introduce several positivity properties:
\begin{itemize}
\item we say $\omega$ is {\it non-collapsing} if there is no bounded $\omega$-plurisubharmonic function $u$ such that 
$(\omega+dd^c u)^n \equiv 0$;
\item $\omega$ satisfies condition $(B)$ if there exists a constant $B>0$ such that
$$
 -B \omega^2 \leq dd^c \omega \leq B \omega^2
 \; \; \text{ and } \; \;
 -B \omega^3 \leq d \omega \wedge d^c \omega \leq B \omega^3;
 $$
 \item we say $\omega$ is {\it  uniformly non-collapsing} if
 $$
v_-(\omega):=\inf \left\{ \int_X (\omega+dd^c u)^n \; : \; u \in\PSH(X,\omega) \cap {L}^{\infty}(X)  \right\} >0.
$$
\end{itemize}

The non-collapsing condition is the minimal positivity condition  one should require.
We show in Proposition \ref{pro:domination} that it implies the {\it domination principle},
a useful extension of the classical maximum principle.
We provide a simple example showing that having positive volume
$\int_X \omega^n>0$ does not prevent from being collapsing (see Example \ref{exa:collapsing}).

After providing a simplified proof of Ko{\l}odziej-Nguyen modified comparison principle
(see \cite[Theorem 0.5]{KN15} and Theorem \ref{thm:pcpr}),
we show that condition (B) implies non-collapsing.
The former condition is e.g. satisfied by any form $\omega$ which is the pull-back
of a Hermitian form on a singular Hermitian variety.

 When $\omega$ is closed, simple integration by parts reveal that $v_-(\omega)=\int_X \omega^n$
is positive as soon as $\omega$ is positive at some point. Bounding from below
$v_-(\omega)$ is a much more delicate issue in general.
We show in Proposition \ref{pro:bddMA2}
that 
$\omega$ is uniformly non-collapsing if
one restricts to $\omega$-psh functions that are uniformly bounded by 
a fixed constant $M$:
$$
v_M^-(\omega):=\inf_{} \left\{ \int_X (\omega+dd^c u)^n  \; : \; 
    u \in\PSH(X,\omega)  \text{ with } -M \leq u \leq 0 \right\}>0.
$$
 
For non uniformly bounded functions we show the following:

\begin{thmA}
{\it 
The condition $v_+(\omega_X)<+\infty$ is independent of the choice of $\omega_X$;
it is moreover invariant under bimeromorphic change of coordinates.

The condition $v_-(\omega_X)>0$ is also independent of the choice of $\omega_X$
and invariant under bimeromorphic change of coordinates.

In particular these conditions both hold true
 if $X$ belongs to the Fujiki class.
}
\end{thmA}

We are not aware of a single example of a compact complex manifold such that $v_+(\omega_X)=+\infty$
or $v_-(\omega_X)=0$. This is an important open problem.

\smallskip

The proof of Theorem A relies on a fine use of quasi-plurisubharmonic envelopes. These envelopes have been systematically studied in \cite{GLZ19} in the K\"ahler framework. 
Adapting and generalizing
 \cite{GLZ19} to this Hermitian setting, we prove in Section \ref{sec:envelopes} the following:

\begin{thmB}
{\it 
Let $\omega$ be a semi-positive $(1,1)$-form.
Given a Lebesgue measurable function $h:X \rightarrow \R$, we define the $\omega$-plurisubharmonic envelope of $h$ by
$
P_{\omega}(h) := \left(\sup \{ u \in \psh (X,\omega) \; :\;  u \leq h   \}\right)^*,
$
where the star means that we take the upper semi-continuous regularization.  
If $h$ is bounded   below, quasi-lower-semi-continuous, and $P_{\omega}(h) <+\infty$, then
\begin{itemize}
\item $P_{\omega}(h)$ is a bounded $\omega$-plurisubharmonic function;
\item $P_{\omega}(h) \leq h$ in $X \setminus P$, where $P$ is pluripolar;
\item $(\omega+dd^c P_{\omega}(h))^n$ is concentrated on the contact  set $\{ P_{\omega}(h)=h\}$.
\end{itemize}
}
\end{thmB}

      \smallskip

      An influential conjecture of Grauert-Riemenschneider \cite{GR70} asked whether 
      the existence of a semi-positive holomorphic line bundle
      $L \rightarrow X$  with $c_1(L)^n>0$
         implies that $X$ is Moishezon
      (i.e. bimeromorphically equivalent to a projective manifold).
           This conjecture has been solved positively by Siu in \cite{Siu84} 
           (with complements by \cite{Siu85} and Demailly \cite{Dem85}).
           
      Demailly and P\u{a}un have proposed a transcendental version of this conjecture
      (see \cite[Conjecture 0.8]{DP04}): given a nef class $\alpha \in H^{1,1}_{BC}(X,\R)$
      with $\alpha^n>0$, they conjectured that $\alpha$ should contain a K\"ahler current,
      i.e. a positive closed $(1,1)$-current which dominates a Hermitian form.
      Recall that the Bott-Chern cohomology group $H_{BC}^{1,1}(X,\R)$
    is the quotient of   closed real smooth $(1,1)$-forms,
   by the image of  ${\mathcal C}^{\infty}(X,\R)$
   under the $dd^c$-operator.
      
      This influential conjecture has been further reinforced by Boucksom-Demailly-P\u{a}un-Peternell 
      who proposed a weak transcendental form of Demailly's holomorphic Morse inequalities \cite[Conjecture 10.1]{BDPP13}.
         This stronger conjecture has been solved recently by Witt-Nystr\"om when $X$ is projective \cite{WN19}.
         
  Building on works of Chiose \cite{Chi13}, Xiao \cite{Xiao15} and Popovici \cite{Pop16}
  we obtain the following answer to the qualitative part of these conjectures:

\begin{thmC}
{\it 
 
	Let $\alpha, \beta \in H^{1,1}_{BC}(X,\mathbb{C})$ be nef classes such that 
 	$
 	\alpha^n > n \alpha^{n-1} \cdot  \beta. 
 	$
 	The following properties are equivalent:
 	\begin{enumerate}
 	\item  $\alpha -\beta$ contains a K\"ahler current;
 	\item $v_+(\omega_X)<+\infty$;
 	\item $X$ belongs to the Fujiki class.
 	\end{enumerate}
}
\end{thmC}

 A consequence of our analysis is that the conjectures of Demailly-P\u{a}un and
 Boucksom-Demailly-P\u{a}un-Peternell can be extended to non closed forms,
 making sense outside the Fujiki class. Progresses in the theory of complex Monge-Amp\`ere equations on compact hermitian manifolds
  have indeed shown  that it is useful to consider
  $dd^c$-perturbations of non closed nef forms.
  It is therefore natural to try and consider an extension of 
  Theorem C. These are the contents of Theorem \ref{thm:dempaun} (when $\beta=0$) and
     Theorem \ref{thm:morse2} (when $\beta \neq 0$).

\begin{ackn} 
We thank Daniele Angella for several useful discussions. 
We thank the referee for a very detailed 
reading   and   numerous useful comments. 
This work has benefited from State aid managed by the   ANR-11-LABX-0040
(research project HERMETIC).
 The authors are also partially supported by the ANR project PARAPLUI.
\end{ackn}

\section{Non collapsing forms}

In the whole article we let $X$ denote a compact complex manifold of complex dimension $n \geq 1$,
and we fix $\omega$ a smooth semi-positive $(1,1)$-form on $X$.

\subsection{Positivity properties}

 \subsubsection{Monge-Amp\`ere operators}

  A function is quasi-plurisub\-harmonic (quasi-psh for short)
  if it is locally given as the sum of  a smooth and a psh function.   
  
 Given an open set $U\subseteq X$, quasi-psh functions
$\f:U \rightarrow \R \cup \{-\infty\}$ satisfying
$
\omega_{\f}:=\omega+dd^c \f \geq 0
$
in the weak sense of currents are called $\omega$-psh functions on $U$. 
Constant functions are $\omega$-psh functions since $\omega$ is semi-positive.
A ${\mathcal C}^2$-smooth function $u\in \mathcal{C}^2(X)$ has bounded Hessian, hence $\e u$ is
$\omega$-psh on $X$ if $0<\e$ is small enough and $\omega$ is  positive (i.e. Hermitian).

\begin{defi}
We let $\PSH(X,\omega)$ denote the set of all $\omega$-plurisubharmonic
 functions which are not identically $-\infty$.  
\end{defi}

The set $\PSH(X,\omega)$ is a closed subset of $L^1(X)$, 
for the $L^1$-topology.
We refer the reader to \cite{Dem12,GZbook,Din16} for basic
properties of $\omega$-psh functions.

\smallskip

We briefly explain how to adapt \cite{BT82} to the Hermitian context, and
refer the reader to \cite{DK12} for a more systematic treatment. 
Let $u$ be a bounded quasi-psh function on $X$, and fix a Hermitian form 
$\omega_X$ such that $\omega_X +dd^c u\geq 0$. Let $U$ be an open subset of $X$ such that 
$\omega_X \leq \beta = dd^c \rho$ for some smooth function $\rho$ in $U$. 
Since $u$ is $\beta$-psh in $U$, 
 $(\beta+dd^c u)^p$  is a well defined closed positive current for $1\leq p\leq k$,
  as follows from Bedford-Taylor's theory. We can thus set
\[
(dd^c u)^k := \sum_{p=0}^k (-1)^p \binom{k}{p} (\beta+dd^c u)^{k-p} \wedge \beta^{p}.  
\]
This is a closed current of bidegree $(k,k)$. 
We claim that $(dd^c u)^k$  
does not depend on the choice of $\beta$. 
To see this, assume $\gamma$ is another K\"ahler form in $U$ that dominates  $\omega_X$ in $U$.
By Demailly's regularization result \cite{Dem92}, we can find  a sequence $(u_j)$
of smooth $\omega_X$-psh functions decreasing to $u$. Then 
\[
\sum_{p=0}^k (-1)^p \binom{k}{p} (\beta+dd^c u_j)^{k-p} \wedge \beta^{p} = (dd^c u_j)^k= \sum_{p=0}^k (-1)^p \binom{k}{p} (\gamma+dd^c u_j)^{k-p} \wedge \gamma^{p}. 
\]
Letting $j\to +\infty$, Bedford-Taylor's convergence theorem \cite{BT82}
ensures that 
\[
\sum_{p=0}^k (-1)^p \binom{k}{p} (\beta+dd^c u)^{k-p} \wedge \beta^{p} = \sum_{p=0}^k (-1)^p \binom{k}{p} (\gamma+dd^c u)^{k-p} \wedge \gamma^{p},
\]
proving the claim.  
Covering $X$ by such small open sets $U_{j}$,
we thus obtain a closed $(k,k)$-current $(dd^c u)^k$ globally defined on $X$. In particular, $(dd^c u)^n$ is a signed measure on $X$ with total mass $0$.

For $u \in \PSH(X,\omega)\cap L^{\infty}(X)$ and $1\leq k \leq n$, we define 
\[
(\omega+dd^c u)^k:= \sum_{p=0}^k \binom{k}{p} (dd^c u)^p \wedge \omega^{k-p}. 
\]
This current is positive.   Indeed when $\omega>0$ is Hermitian we 
can approximate $u$ by smooth $\omega$-psh functions $u_j$
 as above and obtain that $(\omega+dd^c u)^k$ is the limit of the smooth positive forms $(\omega+dd^c u_j)^k$.
The general case follows by approximating $\omega$ by $\omega+\e \omega_X$ and letting  $\varepsilon \to 0$. In particular, $(\omega+dd^c u)^n$ is a positive Radon measure on $X$.

\smallskip

If $u,v$ are bounded $\omega$-psh functions, we define similarly the mixed products
\[
(dd^c u)^k \wedge (dd^c v)^l := \sum_{0\leq p\leq k} \sum_{0\leq q\leq l} (-1)^{p+q} \binom{k}{p}\binom{l}{q} \beta_u^{k-p}\wedge \beta_v^{l-q} \wedge \beta^{k+l-p-q},
\]
where $\beta_u:=\beta+dd^c u$, and $\beta$ is a K\"ahler form that dominates $\omega$ in   $U$. 

The mixed Monge-Amp\`ere measures 
\[
\omega_u^j\wedge \omega_v^{n-j} := (\omega+dd^c u)^j \wedge (\omega+dd^c v)^{n-j}
\]
 are thus similarly well defined as positive Radon measures for any $0 \leq j \leq n$, and any bounded $\omega$-psh functions $u,v$.
 
 \subsubsection{Maximum principle}

We will need the following   maximum principle.

\begin{lemma}\label{lem:Demkey}
	Let $u,v, \varphi$ be bounded $\omega$-psh functions in $U\subseteq X$.  
	For all $1\leq j\leq n$, 
	\begin{equation}\label{eq:maxprin}
	{\bf 1}_{\{u>v\}}\omega_{\max(u,v)}^j \wedge \omega_{\varphi}^{n-j} = {\bf 1}_{\{u>v\}} \omega_u^j \wedge \omega_{\varphi}^{n-j},
	\end{equation}
	and
	\begin{equation}\label{eq:Dem inequality}
		\omega_{\max(u,v)}^j \wedge \omega_{\varphi}^{n-j} 
		 \geq {\bf 1}_{\{u>v\}}\omega_u^j \wedge \omega_{\varphi}^{n-j}
		+ {\bf 1}_{\{u\leq v\}}\omega_v^j \wedge \omega_{\varphi}^{n-j}.
	\end{equation}
	If moreover $u\geq v$, then $ {\bf 1}_{\{u= v\}}\omega_u^j \wedge \omega_{\varphi}^{n-j}  \geq {\bf 1}_{\{u=v\}}\omega_v^j \wedge \omega_{\varphi}^{n-j}$.
\end{lemma}
\begin{proof}
	We prove \eqref{eq:maxprin}. Since it is a local property, we can assume that there is a K\"ahler form $\beta$ in $U$ such that $\beta\geq  \omega$. We write $\beta_u$ for $\beta +dd^c u$. By definition, the left-hand side is equal to 
	\[
	{\bf 1}_{\{u>v\}} \sum_{p,q}  \binom{j}{p} \binom{n-j}{q}(-1)^{n-p-q} \beta _{\max(u,v)}^{j-p} \wedge \beta_{\varphi}^{n-j-q}\wedge  (\beta-\omega)^{p+q}. 
	\]
	The result thus follows from the maximum principle of Bedford-Taylor: we can use  \cite[Theorem 3.27]{GZbook} for $T= \beta_{\varphi}^{n-j-q}$
	(in \cite[Theorem 3.27]{GZbook}, the current $T$ is of bidegree $(n-p,n-p)$ but the same proof works here without modification). 
	
	To prove \eqref{eq:Dem inequality} we apply \eqref{eq:maxprin} to $\max(u,v+\varepsilon)$:
	\begin{flalign*}
	\omega_{\max(u,v+\varepsilon)}^j \wedge \omega_{\varphi}^{n-j} & \geq {\bf 1}_{\{u>v+\varepsilon\}}\omega_u^j \wedge \omega_{\varphi}^{n-j}
		+ {\bf 1}_{\{v +\varepsilon>u\}}\omega_v^j \wedge \omega_{\varphi}^{n-j}\\
		&\geq  {\bf 1}_{\{u>v+\varepsilon\}}\omega_u^j \wedge \omega_{\varphi}^{n-j}
		+ {\bf 1}_{\{v \geq u\}}\omega_v^j \wedge \omega_{\varphi}^{n-j}.
	\end{flalign*}
	Since the sets $\{u>v+\varepsilon\}$ increase to $\{u>v\}$ as $\varepsilon \searrow 0$, letting $\varepsilon \to 0$ we obtain the second statement.

	 The last statement follows by multiplying \eqref{eq:Dem inequality} with ${\bf 1}_{\{u=v\}}$. 
\end{proof}

\subsubsection{Condition (B) and non-collapsing}
We  always assume in this article that   $\int_X \omega^n>0$. 
 On a few occasions we will need to assume  positivity properties
 that are possibly slightly stronger:

\begin{defi}
We say   $\omega$ satisfies condition (B) if there exists
 $B \geq 0$ such that
 $$
 -B \omega^2 \leq dd^c \omega \leq B \omega^2
 \; \; \text{ and } \; \;
 -B \omega^3 \leq d \omega \wedge d^c \omega \leq B \omega^3.
 $$
\end{defi}

Here are three different contexts where this condition is satisfied:
\begin{itemize}
\item any Hermitian metric $\omega>0$ satisfies condition (B);
\item if $\pi:X \rightarrow Y$ is a desingularization of a singular compact complex variety $Y$
and  $\omega_Y$ is a Hermitian metric, 
then $\omega=\pi^*\omega_Y$ satisfies condition (B);
\item if $\omega$ is semi-positive and closed, then it satisfies condition (B).
\end{itemize}
Combining these, one obtains further settings where condition (B) is satisfied.

\begin{defi}
We say   $\omega$ is non-collapsing if 
 for any bounded $\omega$-psh function, the complex Monge-Amp\`ere 
 measure $(\omega+dd^c u)^n$ has positive mass: $\int_X \omega_u^n>0$.
\end{defi}

We shall   see in Corollary \ref{cor:Bvsnoncollapsing} below  that condition (B) implies non-collapsing.

 \subsection{Comparison principle}

 The comparison principle plays a central role in K\"ahler pluripotential theory.
 Its proof breaks down in the Hermitian setting, as it heavily relies on the closedness of 
 the reference form $\omega$ through the preservation of Monge-Amp\`ere masses. 
 In that context the following 
 ``modified comparison principle'' has been established by Ko{\l}odziej-Nguyen 
 \cite[Theorem 0.2]{KN15}:
 
 \begin{theorem} \label{thm:pcpr}
 Assume $\omega$ satisfies condition $(B)$ and let $u,v$ be bounded $\omega$-psh functions. 
 For $\lambda \in (0,1)$ we set $m_{\lambda}=\inf_X \{ u-(1-\lambda) v \}$.
 Then 
 $$
 \left( 1-\frac{4B(n-1)^2s}{\lambda^3} \right)^n \int_{ \{u<(1-\lambda)v+m_{\lambda}+s\} } \omega^n_{(1-\lambda)v}
 \leq \int_{ \{u<(1-\lambda)v+m_{\lambda}+s\} } \omega^n_{u}.
 $$
 for all $0<s< \frac{\lambda^3}{32B(n-1)^2}$.
 \end{theorem}
 
 The   proof by Ko{\l}odziej-Nguyen relies on the main result of \cite{DK12}, together
 with  extra fine estimates.
 We propose here a simplified proof.
 
 \begin{proof}
 Set $\phi:= \max(u, (1-\lambda) v+ m_{\lambda} +s)$, $U_{\lambda,s}:= \{u< (1-\lambda)v + m_{\lambda}+s\}$. 
 For $0\leq k\leq n$ we set $T_k:= \omega_{u}^k\wedge \omega_{\phi}^{n-k}$, and $T_l=0$ if $l<0$.  
 Set $a= Bs\lambda^{-3}(n-1)^{2}$.
 We are going to prove by induction on $k=0,1,...,n-1$ that
\begin{equation}\label{eq: induction}
(1-4a)\int_{U_{\lambda,s}}T_k  \leq  \int_{U_{\lambda,s}}T_{k+1}. 
\end{equation}
The conclusion follows since $(\omega_{\phi})^n=(\omega_{(1-\lambda)v})^n$ in 
the plurifine open set $U_{\lambda,s}$.

\smallskip

We first prove \eqref{eq: induction} for $k=0$. 
Since $u\leq \phi$,  Lemma \ref{lem:Demkey} ensures that
$$
{\bf 1}_{\{u=\phi\}} \omega_{\phi}^{n} \geq {\bf 1}_{\{u=\phi\}} \omega_u\wedge \omega_{\phi}^{n-1}.
$$
 Observing that $U_{\lambda,s}= \{u<\phi\}$ we infer
\[
 \int_X dd^c (\phi-u) \wedge \omega_{\phi}^{n-1}=
\int_X (\omega_{\phi} - \omega_{u}) \wedge \omega_{\phi}^{n-1}\geq  \int_{U_{\lambda,s}} \omega_{\phi}^n -  \int_{U_{\lambda,s}} \omega_u \wedge \omega_{\phi}^{n-1}. 
\]
A direct computation shows that
\begin{eqnarray*}
dd^c \omega_{\phi}^{n-1} &= &(n-1)dd^c \omega \wedge \omega_{\phi}^{n-2}+
(n-1)(n-2) d \omega \wedge  d^c \omega \wedge \omega_{\phi}^{n-3} \\
&\leq &(n-1)B \omega^2 \wedge \omega_{\phi}^{n-2}+
(n-1)(n-2) B \omega^3 \wedge \omega_{\phi}^{n-3},
\end{eqnarray*}
since $\omega$ satisfies condition (B).
As $\phi-u \geq 0$, it follows from Stokes theorem that
$$
 \int_X dd^c (\phi-u) \wedge \omega_{\phi}^{n-1} \leq 
 (n-1)B \left\{ \int_X (\phi-u) \omega^2 \wedge \omega_{\phi}^{n-2}+
(n-2)  \int_X (\phi-u) \omega^3 \wedge \omega_{\phi}^{n-3} \right\}.
$$
Observe that
\begin{itemize}
\item  $\lambda \omega \leq \omega_{(1-\lambda)v}$ hence 
$\omega^{j} \wedge \omega_{\phi}^{n-j} \leq \lambda^{-j} (\omega_{(1-\lambda)v})^j \wedge \omega_{\phi}^{n-j} $,
\item $ (\omega_{(1-\lambda)v})^{j}  \wedge \omega_{\phi}^{n-j} = \omega_{\phi}^n$
in the plurifine open set $U_{\lambda,s}$,
\item $0\leq \phi-u \leq s$ and $\phi-u=0$ on $X\setminus U_{\lambda,s}$,
\end{itemize}
to conclude that $\int_X (\phi-u) \omega^j \wedge \omega_{\phi}^{n-j} \leq s \lambda^{-j} \int_{U_{\lambda,s}} \omega_{\phi}^n$,
for $j=2,3$, 
hence
\[
\int_{U_{\lambda,s}} \omega_{\phi}^n - \int_{U_{\lambda,s}} \omega_{u} \wedge \omega_{\phi}^{n-1}  \leq 
 \int_X dd^c (\phi-u) \wedge \omega_{\phi}^{n-1} 
\leq \frac{Bs(n-1)^2}{\lambda^3}\int_{U_{\lambda,s}} \omega_{\phi}^n,
\]
since $\lambda^{-2} \leq \lambda^{-3}$.
This yields \eqref{eq: induction} for $k=0$.  

\smallskip

We asume now that \eqref{eq: induction} holds for all $j \leq k-1$, and we check that it still holds for $k$.  
Observe that
\begin{eqnarray*}
\lefteqn{dd^c \left( \omega_u^k \wedge \omega_{\phi}^{n-[k+1]} \right)  } \\
& =&  k dd^c \omega \wedge \omega_u^{k-1} \wedge \omega_{\phi}^{n-[k+1]}
+(n-[k+1]) dd^c \omega \wedge \omega_u^k \wedge \omega_{\phi}^{n-[k+2]} \\
&+& 2 k(n-[k+1]) d \omega \wedge d^c \omega \wedge \omega_u^{k-1} \wedge \omega_{\phi}^{n-[k+2]} \\
&+&k(k-1) d \omega \wedge d^c \omega \wedge \omega_u^{k-2} \wedge \omega_{\phi}^{n-[k+1]} \\
&+&  (n-[k+1])[n-(k+2)] d \omega \wedge d^c \omega \wedge \omega_u^{k} \wedge \omega_{\phi}^{n-[k+3]}.
\end{eqnarray*}
The same arguments as above therefore show that 
\begin{flalign*}
 \int_{U_{\lambda,s}} (T_k&-T_{k+1}) \leq  \int_X (T_k-T_{k+1}) = \int_X (\phi-u) dd^c (\omega_{u}^{k}\wedge \omega_{\phi}^{n-[k+1]}) \\
&\leq \frac{Bs}{\lambda^3} \int_{U_{\lambda,s}} \left(k(k-1)T_{k-2}+2k[n-k]T_{k-1}+(n-[k+1])^2 T_{k} \right) \\
&\leq a \left ( \frac{1}{(1-4a)^2} + \frac{1}{1-4a}+1 \right ) \int_{U_{\lambda,s}}T_k
\leq 4a   \int_{U_{\lambda,s}} T_k,
\end{flalign*}
where in the third inequality above we have used the induction hypothesis,
while the fourth inequality follows from the upper bound $4a<1/8$.
 From this we obtain \eqref{eq: induction} for $k$, finishing the proof. 
 \end{proof}

 \begin{coro} \label{cor:Bvsnoncollapsing}
 	If $\omega$ satisfies condition $(B)$	then $\omega$ is non-collapsing. 
 \end{coro}

\begin{proof}
It follows from Theorem \ref{thm:pcpr} that the domination principle holds
(see \cite[Proposition 2.2]{LPT20}).
The latter implies in particular that if $u,v$ are $\omega$-psh and bounded,
then
 $
 e^{- v}(\omega+dd^c v)^n \geq e^{-u} (\omega+dd^c u)^n \Longrightarrow v \leq u
$
(see \cite[Proposition 2.3]{LPT20}).
There can thus be no bounded $\omega$-psh function $u$ such that $(\omega+dd^c u)^n=0$.
Otherwise the previous inequality applied with a constant function $v=A$ yields
$u \geq A$ for any $A$, a contradiction.
\end{proof}

\section{Envelopes} \label{sec:envelopes}

We consider here envelopes of $\omega$-psh functions, extending
 some results of \cite{GLZ19} that have been established for K\"ahler manifolds.

\subsection{Basic properties}

 \begin{defi}
   A Borel set $E\subset X$ is (locally) pluripolar if  it is locally contained in the $-\infty$ locus of some psh function: for each $x\in X$, there exists an open neighborhood $U$ of $x$ and $u\in\PSH(U)$ such that $E\cap U\subset \{u=-\infty\}$.
 \end{defi}

\begin{defi} \label{def:usual}
Given a Lebesgue measurable function $h:X \rightarrow \R$, we define the $\omega$-psh envelope of $h$ by
$$
P_{\omega}(h) := \left(\sup \{ u \in \psh (X,\omega)  \; : \;  u \leq h  \, \, \text{quasi-everywhere in} \, \, X\}\right)^*,
$$
where the star means that we take the upper semi-continuous regularization, 
while quasi-everywhere means outside a locally pluripolar set.
\end{defi}

 When $\omega$ is Hermitian and $h$ is ${\mathcal C}^{1,1}$-smooth, then so is $P_{\omega}(h)$ 
 (see \cite{Ber19,CZ19,CM20}) and one can   show that
 \begin{equation} \label{eq:MAenv}
 (\omega+dd^c P_{\omega}(h))^n={\bf 1}_{\{ P_{\omega}(h)=h\}} (\omega+dd^c h)^n.
 \end{equation}
 For merely semipositive $\omega$ and less regular obstacle $h$ we have the following:

 \begin{thm} \label{thm:orthog}
If $h$ is bounded from  below, quasi-l.s.c., and $P_{\omega}(h) <+\infty$, then
\begin{itemize}
\item $P_{\omega}(h)$ is a bounded $\omega$-plurisubharmonic function;
\item $P_{\omega}(h) \leq h$ in $X \setminus P$, where $P$ is pluripolar;
\item $(\omega+dd^c P_{\omega}(h))^n$ is concentrated on the contact  set $\{ P_{\omega}(h)=h\}$.
\end{itemize}
\end{thm}

Recall that a function $h$ is quasi-lower-semicontinuous (quasi-l.s.c.) if for any $\e>0$,
there exists an open set $G$ of capacity smaller than $\e$ such that
$h$ is continuous in $X \setminus G$. Quasi-psh functions are quasi-continuous (see \cite{BT82}), 
as well as differences of the latter.

\begin{proof}
 The proof  is an adaptation of \cite[Proposition 2.2, Lemma 2.3, Proposition 2.5]{GLZ19}, which deal
 with the case when $\omega$ is K\"ahler. 
 
 Since $P_{\omega}(h)$ is bounded from above, up to replacing $h$ with $\min(h, C)$ with $C>\sup_X P_{\omega}(h)$ we can assume that $h$ is bounded. 
 
\smallskip

\noindent	{\it Step 1: $h$ is smooth, $\omega$ is Hermitian}. 
Building on Berman's work \cite{Ber19}, it was shown by 
Chu-Zhou in \cite{CZ19} that the smooth solutions $\varphi_{\beta}$ to 
	\[
	(\omega+dd^c \varphi_{\beta})^n = e^{\beta(\varphi_{\beta}-h)} \omega^n 
	\]
	converge uniformly to $P_{\omega}(h)$ along with uniform ${\mathcal C}^{1,1}$-estimates. 
	As a consequence,  the measures $(\omega+dd^c \varphi_{\beta})^n$ converge weakly to $(\omega+dd^c P_{\omega}(h))^n$. For each fixed $\varepsilon>0$, we have the inclusions of open sets $\{P_{\omega}(h)<h-2\varepsilon\} \subset \{\varphi_{\beta}<h -\varepsilon\}$ for $\beta$ large enough, yielding  
	\begin{flalign*}
	\int_{\{P_{\omega}(h)<h-2\varepsilon\}} (\omega+dd^c P_{\omega}(h))^n &\leq \liminf_{\beta\to +\infty}  \int_{\{P_{\omega}(h)<h-2\varepsilon\}}	 (\omega+dd^c \varphi_{\beta})^n 	\\
	&\leq  \liminf_{\beta\to +\infty}  \int_{\{P_{\omega}(h)<h-2\varepsilon\}}	 e^{-\beta \varepsilon} \omega^n =0. 
	\end{flalign*}
	
\smallskip
   
\noindent	{\it	Step 2: $h$ is lower semi-continuous, $\omega$ is Hermitian.}
 If $h$ is continuous, we can approximate it uniformly by smooth functions $h_j$. Letting $u_j: = P(h_j)$ the previous step ensures that
	\[
	\int_X (h_j-u_j) (\omega+dd^c u_j)^n =0. 
	\]
	As $h_j \to h$ uniformly we also have that $u_j \to u:=P(h)$ uniformly 
	and the desired property follows from Bedford-Taylor's convergence theorem.

 When $h$ is merely  lower semi-continuous, we let
 $(h_j)$ denote a sequence of continuous functions which increase pointwise to $h$ and set $u_j =P(h_j)$. 
 Then $u_j \nearrow u$ a.e. on $X$ for some bounded function $u\in\PSH(X,\omega)$. Since $u_j\leq h_j \leq h$ quasi-everywhere on $X$ we infer $u\leq h$ quasi-everywhere on $X$, hence $u\leq P(h)$. For each $k<j$, the second step  ensures that 
	\[
	\int_{\{u<h_k\}} (\omega+dd^c u_j)^n \leq \int_{\{u_j<h_j\}} (\omega+dd^c u_j)^n=0.
	\]
	Since $\{u<h_k\}$ is open, letting $j\to +\infty$ and then $k\to +\infty$ we arrive at 
	$$\int_{\{u<h\}} (\omega+dd^c u)^n=0.$$ 
	We also have that $P(h) \leq h$ quasi-everywhere on $X$, hence 
	$$
	\int_{\{u<P(h)\}} (\omega+dd^c u)^n=0,
	$$
	and \cite[Proposition 2.2]{LPT20} then ensures that $u=P(h)$.

\smallskip
   
\noindent	{\it	Step 3: $h$ is quasi-l.s.c., $\omega$ is Hermitian.}
 By \cite[Lemma 2.4]{GLZ19} we can find a decreasing sequence $(h_j)$ of lsc functions such that $h_j \searrow h$ q.e. on $X$ and $h_j \to h$ in capacity. Then $u_j := P(h_j) \searrow u:=P(h)$.  By Step 2  we know that for all $j>k$, 
	$$
	\int_{\{u_k<h\}} (\omega+dd^c u_j)^n \leq \int_{\{u_j<h_j\}} (\omega+dd^c u_j)^n =0. 
	$$
	Since $\{u_k<h\}$ is quasi-open and the functions $u_j$ are uniformly bounded, letting $j\to +\infty$  we obtain
	$$
	\int_{\{u_k <h\}} (\omega+dd^c u)^n =0.
	$$
Letting $k\to +\infty$  yields the desired result. 

\smallskip
   
\noindent	{\it	Step 4: the general case.}
We approximate $\omega \geq 0$ by the Hermitian forms $\omega_{j}=\omega+j^{-1} \omega_X>0$.
Observe that $j \mapsto u_{j}=P_{\omega_{j}}(h)$ decreases to $u=P_{\omega}(h)$ as $j$ increases to $+\infty$.
For $0<k<j$, the previous step ensures that
	\[
	\int_{\{u_k<h\}} (\omega+j^{-j} \omega_X+dd^c u_j)^n =0. 
	\]

	As the set $\{u_k<h\}$ is quasi-open and $u_j$ is uniformly bounded we can let $j\to +\infty$ and use 
	Bedford-Taylor's convergence theorem to get 
	\[
	\int_{\{u_k <h\}} (\omega+dd^c u)^n =0,
	\]
	We finally let $k\to +\infty$ to conclude.
\end{proof}

For later use we extend the latter result to a setting where
$P_{\omega}(h)$ is not necessarily globally bounded:
 
 \begin{coro}
 	If $h$ is quasi-lower-semicontinuous and $P_{\omega}(h)$ is locally bounded in a non-empty open set $U\subset X$ then $(\omega+dd^c P_{\omega}(h))^n$ is a well-defined positive Borel measure in $U$ which vanishes in $U\cap \{P_{\omega}(h) <h\}$. 
 \end{coro}
 
 \begin{proof}
 	Let $(h_j)$ be a sequence of l.s.c. functions decreasing to $h$ quasi-everywhere.
 	 Then $u_j:=P_{\omega}(h_j)$ is a bounded  $\omega$-psh function such that $(\omega+dd^c u_j)^n =0$ on $\{u_j<h_j\}$.  
 	Since $u_j$ decreases to $u:= P_{\omega}(h)$, Bedford-Taylor's convergence theorem ensures that $\omega_{u_j}^n \to \omega_u^n$ in $U$.
 	
 	 Fix $U'$ a relatively compact open set $U'\Subset U$. For each $k$ fixed the set $\{u_k <h\}$ is quasi open 
 	 and the functions $u_j,u$ are uniformly bounded in $U'$, hence 
 	\[
 	\liminf_{j\to +\infty} \int_{\{u_k <h\}\cap U'} \omega_{u_j}^n \geq \int_{\{u_k <h\}\cap U'} \omega_u^n,
 	\]
 	which implies, after letting $k\to +\infty$, that $\omega_u^n$ vanishes in $U' \cap \{u<h\}$. 
 	We finally let $U'$ increase to $U$ to conclude.
 \end{proof}

 We shall use later on the following:

\begin{lem} \label{lem:envmin}
Let $u,v$ be bounded $\omega$-psh functions,  
and let $\varphi=P_{\omega}(\min(u,v))$ denote the $\omega$-psh envelope of $\min(u,v)$. 
Then
\begin{enumerate}
\item $(\omega+dd^c \varphi)^n \leq {\bf 1}_{\{\varphi =u<v\}}\omega_u^n+{\bf 1}_{\{\varphi =v\}}\omega_v^n$;

\item if $(\omega+dd^c u)^n=fdV_X$ and $(\omega+dd^c v)^n=gdV_X$, then
$$
(\omega+dd^c \varphi)^n \leq \max(f,g) dV_X,
$$
while
$$
(\omega+dd^c \varphi )^n \geq \min(f,g) dV_X.
$$
\end{enumerate}
\end{lem}

\begin{proof}
Since $\min(u,v)$ is quasi-continuous, it follows from Theorem \ref{thm:orthog} that
the Monge-Amp\`ere measure  $\omega_{\varphi}^n$   is supported on 
\[
\{ \varphi=\min (u,v) \} \subset \{ \varphi=u\} \cup \{ \varphi=v \}.
\]
Thus
\begin{equation} \label{eq:MAmin}
\omega_{\f}^n \leq {\bf 1}_{\{\f=u<v\}} \omega_{\f}^n+{\bf 1}_{\{\f=v\}}\omega_{\f}^n.
\end{equation}
Since $\f=P(\min(u,v)) \leq u$ and $\f=P(\min(u,v)) \leq v$,  
Lemma \ref{lem:Demkey} yields
$$
{\bf 1}_{\{\f= u \}} \omega_{\f}^n \leq {\bf 1}_{\{\f = u \}}\omega_u^n, \; \text{and}\; {\bf 1}_{\{\f= v \}} \omega_{\f}^n  \leq  {\bf 1}_{\{\f= v \}} \omega_v^n.
$$
 Together with \eqref{eq:MAmin} we obtain (1). 
 
 When $(\omega+dd^c u)^n=fdV_X$ and $(\omega+dd^c v)^n=gdV_X$, we obtain
 $$
 {\bf 1}_{\{\f=u<v\}} \omega_{\f}^n \leq {\bf 1}_{\{\f=u<v\}}  f dV_X \leq  {\bf 1}_{\{\f=u<v\}}  \max(f,g) dV_X
 $$
 and  ${\bf 1}_{\{\f=v\}} \omega_{\f}^n \leq {\bf 1}_{\{\f=v\}}  g dV_X \leq  {\bf 1}_{\{\f=v\}}  \max(f,g) dV_X$, hence
 $$
  \omega_{\f}^n \leq \left\{ {\bf 1}_{\{\f=u<v\}} + {\bf 1}_{\{\f=v\}} \right\}  \max(f,g) dV_X \leq  \max(f,g) dV_X.
 $$
 
The last item follows from
 Lemma \ref{lem:Demkey}.
\end{proof}

 \subsection{Locally vs globally pluripolar sets}

 A classical result of Josefson asserts that a locally pluripolat set $E$ in $\C^n$ is {\it globally pluripolar}, i.e.
 there exists a   psh function $u \in\PSH(\C^n)$ such that $E \subset \{u=-\infty\}$.
 This result has been extended to compact K\"ahler manifolds in \cite{GZ05}, and
 to the Hermitian setting in \cite{Vu19}:  if $E\subset X$ is locally pluripolar and
  $\omega_X$ is a Hermitian form, one can find $u \in\PSH(X,\omega_X)$ such that
 $E \subset \{ u=-\infty\}$.
 
 \smallskip
 
We further extend this result to the case of non-collapsing forms:
 
 \begin{lemma}\label{lem: Josefson semipositive}
If $E$ is (locally) pluripolar and $\omega\geq 0$ is non-collapsing then $E\subset \{u=-\infty\}$ for some $u \in\PSH(X,\omega)$. 
\end{lemma}

 The proof is a consequence of Theorem \ref{thm:orthog} and analogous results established on K\"ahler manifolds.

\begin{proof}
	As in \cite[Theorem 5.2]{GZ05} it is enough to check that $V_{E,\omega}^* \equiv +\infty$, where
	$$
	V_{E,\omega}(x)=\sup \{ \f(x)\; :\;  \f \in\PSH(X,\omega) \text{ and } \f \leq 0\; \text{quasi-everywhere on }  E \}.
	$$
Here quasi-everywhere means outside a locally pluripolar set. 	
	 If it is not the case then $V_{E,\omega}^*$ is a bounded $\omega$-psh function on $X$. We can assume that $E\subset U\Subset V \Subset  V'$ is contained in a holomorphic chart $V'$. By Josefson's theorem (see \cite[Theorem 4.4]{GZbook}) we can find $u\in L^1_{{\rm loc}}(V')$  a psh function in $V'$ such that $E\subset \{u=-\infty\}$. Let $u_j$ be a sequence of smooth psh functions in a neighborhood of $V$ such that $u_j \searrow u$. 
	 Fix $N\in \N$ and for $j$ large enough we set 
	\[
	K_{j,N}:= \{x\in V \; : \; u_j(x) \leq -N\}, \; \varphi_{j,N}: =V_{K_{j,N},\omega}^*, 
	\]
	and note that $K_{j,N}$ is increasing in $j$,  $\varphi_{j,N} \searrow \varphi_N \in \PSH(X,\omega) \cap L^{\infty}(X)$ as $j\to +\infty$. We also have that $E \subset \cup_{j\geq 1} K_{j,N}$, hence $0\leq \varphi_N \leq V_{E,\omega}^*$.  We can thus find $j_N$ so large that $\varphi_{j,N} \leq \sup_X V_{E,\omega}^*+1$ for all $j\geq j_N$.
	 
	Let $\rho$ be a smooth psh function in $V$ such that $dd^c \rho \geq \omega$. 
	The Chern-Levine-Nirenberg inequality (see \cite[Theorem 3.14]{GZbook}) ensures that, for $j\geq j_N$, 
	\begin{flalign*}
		\int_{K_{j,N}} (\omega+dd^c \f_{j,N})^n & \leq \int_{K_{j,N}} (dd^c (\rho+\f_{j,N}))^n\\
		& \leq \frac{1}{N} \int_{V} |u_j| (dd^c (\rho+\f_{j,N}))^n\\
		& \leq  \frac{C}{N},
	\end{flalign*}
	for some uniform constant $C>0$. The function which is zero on ${K_{j,N}}$ and $+\infty$ elsewhere is lower semi-continuous on $X$ since $K_{j,N}$ is compact. It thus follows from Theorem \ref{thm:orthog} that
	\[
	\int_X (\omega+dd^c \f_{j,N})^n=\int_{K_{j,N}} (\omega+dd^c \f_{j,N})^n \leq \frac{C'}{N}.
	\]
	Letting $j\to +\infty$ we obtain $\int_X (\omega+dd^c \f_N)^n \leq C'/N$. Now $\f_N\nearrow \varphi$ 
	as $N \rightarrow +\infty$, for some $\varphi \in\PSH(X,\omega)$ which is bounded since $0\leq \varphi_N\leq V_{E,\omega}^*$. 
We thus obtain $\int_X (\omega+dd^c \varphi)^n =0$,  yielding a contradiction
since $\omega$ is non-collapsing and $\varphi$ is bounded. 
\end{proof}

 Since locally pluripolar sets are $\PSH(X,\omega)$-pluripolar, arguing as in the proof of \cite[Proposition 2.2]{GLZ19}, one finally obtains:

\begin{coro}
	Let $f$ be a Borel function
	such that $P_{\omega}(f)\in\PSH(X,\omega)$. Then 
	\[
	P_{\omega}(f) =  \left(\sup \{ u \in \psh (X,\omega)  \; :\;  u \leq f \, \, \text{in} \, \, X\}\right)^*.
	\] 
\end{coro}

  \subsection{Domination principle}

We now establish the following  generalization of the domination principle:

  \begin{prop} \label{pro:domination}
  
  Assume $\omega$ is non-collapsing and fix $c\in [0,1)$.
  If  $u,v$ are bounded $\omega$-psh functions such that $\omega_u^n \leq c \omega_v^n$ on $\{u<v\}$,
  then $u\geq v$. 
  \end{prop}
  
  The usual domination principle corresponds to the case $c=0$
  (see \cite[Proposition 2.2]{LPT20}).
 
 \begin{proof}
 	Fixing $a>0$ arbitrarily small, we are going to prove that $u\geq v-a$ on $X$. 
 	Assume by contradiction that   $E=\{u<v-a\}$ is not empty. Since $u,v$ are quasi-psh, the set $E$ 
 	has positive Lebesgue measure. 	For $b>1$ we set 
 	$$
 	u_b := P_{\omega}(bu-(b-1)v).
 	$$ 
 	 It follows from Theorem \ref{thm:orthog}  that $(\omega+dd^c u_b)^n$ is concentrated on the   set 
 	 $$
 	 D:=\{u_b =bu-(b-1)v\}.
 	 $$
 	Note also that $b^{-1}u_b+(1-b^{-1})v \leq u$ with equality on $D$. Therefore 
	\[
	{\bf 1}_D(\omega +dd^c (b^{-1}u_b+(1-b^{-1})v) )^n  \leq {\bf 1}_D \omega_u^n,
	\]
as follows from Lemma \ref{lem:Demkey}, 
hence
	\[
	{\bf 1}_D b^{-n}(\omega +dd^c u_b)^n +{\bf 1}_D(1-b^{-1})^n (\omega+dd^c v)^n   \leq {\bf 1}_D \omega_u^n. 
	\]
	
We choose $b$ so large that $(1-b^{-1})^n >c$. Multiplying the above inequality by ${\bf 1}_{\{u<v\}}$ 
and noting that $\omega_u^n \leq c\omega_v^n$ on $\{u<v\}$, we obtain 
	\[
	{\bf 1}_{D\cap \{u<v\}} (\omega +dd^c u_b)^n =0.  
	\]
Since $u_b$ is bounded and $\omega$ is non-collapsing, we know that $\omega_{u_b}^n(D)=\omega_{u_b}^n(X) >0$. 
We infer that the set $D\cap \{u\geq v\}$ is not empty, and on this set we have
$$
u_b=bu-(b-1)v \geq u \geq -C,
$$
 since $u$ is bounded. It thus follows that $\sup_X u_b$ is uniformly bounded from below. 
		As $b \rightarrow +\infty$ the functions $u_b-\sup_X u_b$   converge to a function $u_{\infty}$
	which is 	$-\infty$ on $E$, but not identically $-\infty$ hence it belongs to $\PSH(X,\omega)$. This implies that the set $E$ has Lebesgue measure $0$, a contradiction. 
 \end{proof}
 
 Here is a direct consequence of the domination principle:

 \begin{coro} \label{cor:pcp1}
   Assume $\omega$ is non-collapsing 
and let $u,v$ be bounded $\omega$-psh functions. 
 Then for all $\e>0$,
 $$
 e^{-\e v}(\omega+dd^c v)^n \geq e^{-\e u} (\omega+dd^c u)^n \Longrightarrow v \leq u.
$$
 \end{coro}
 
 \begin{proof}
 Fix $a>0$. On the set $\{u<v-a\}$ we have $\omega_u^n \leq e^{-\varepsilon a} \omega_v^n$. 
 Proposition \ref{pro:domination} thus gives $u\geq v-a$. This is true for all $a>0$, hence $u\geq v$. 
 \end{proof}

   \section{Bounds on Monge-Amp\`ere masses} \label{sec:mass}
   
    In the sequel we fix  a Hermitian form $\omega_X$ on $X$.

     \subsection{Global bounds}
 
  Since the semi-positive $(1,1)$-form $\omega$ is not necessarily 
     closed, the mass of the complex Monge-Amp\`ere measures
    $(\omega+dd^c u)^n$ is (in general) not   constantly equal to $V_{\omega}:=\int_X \omega^n>0$.

    \begin{definition}
 For $1 \leq j \leq n$ we consider
      $$
    v_{-,j}(\omega):=\inf \left\{ \int_X (\omega+dd^c u)^j \wedge \omega^{n-j} \; : \; u \in\PSH(X,\omega) \cap {L}^{\infty}(X)  \right\}
    $$
   and
    $$
    v_{+,j}(\omega):=\sup \left\{ \int_X (\omega+dd^c u)^j \wedge \omega^{n-j}  \; :\; u \in\PSH(X,\omega) \cap {L}^{\infty}(X) \right\}.
    $$
    We set $v_-(\omega):=v_{-,n}(\omega)$ and $v_{+}(\omega)=v_{+,n}(\omega)$. When $\omega>0$ is Hermitian, the supremum and infimum in the definition of $v_{+,j}(\omega)$ and $v_{-,j}(\omega)$ can be taken  over $\PSH(X,\omega) \cap C^{\infty}(X)$ as follows from Demailly's approximation \cite{Dem92} and Bedford-Taylor's convergence theorem \cite{BT76,BT82}.  
    \end{definition}

     It is  an interesting open problem to determine when $v_-(\omega_X)$ is   positive
     and/or $v_+(\omega_X)$ is  finite. 
     These conditions may depend on the complex structure, but they are independent of the choice
     of Hermitian metric.

     \subsubsection{Monotonicity and invariance properties}

    \begin{prop} \label{pro:bddsMA}
    Let $0 \leq \omega_1 \leq \omega_2$ be semi-positive $(1,1)$-forms. Then
  \begin{equation} \label{eq:v-+}
    v_-(\omega_1) \leq v_-(\omega_2) 
   \; \; \text{ and } \; \;
    v_{+,j}(\omega_1) \leq v_{+,j}(\omega_2),
    \end{equation}
    for all $1\leq j\leq n$. 
  Moreover
    
    1)   $ v_{+,j}(\omega_X)<+\infty \Longleftrightarrow v_{+,j}(\omega_X')<+\infty$ for any other Hermitian metric $\omega_X'$.
     
     \smallskip
     
      2)  $0< v_-(\omega_X) \Longleftrightarrow 0< v_-(\omega_X')$ for any other Hermitian metric $\omega_X'$.    
     \end{prop}

      \begin{proof}
        Since any $\omega_1$-psh function $u$ is also $\omega_2$-psh, we obtain
    $$
    \int_X (\omega_1+dd^c u)^j \wedge \omega_1^{n-j} \leq \int_X (\omega_2+dd^c u)^j \wedge \omega_2^{n-j} \leq v_{+,j}(\omega_2).
    $$
    which shows that $v_{+,j}(\omega_1) \leq v_{+,j}(\omega_2)$.
          We now bound $v_-(\omega_2)$ from below. Let $v$ be a bounded $\omega_2$-psh function
      and let $u=P_{\omega_1}(v)$ denote its $\omega_1$-psh envelope. Then $u$ is a bounded $\omega_1$-psh function and $u\leq v$ on $X$. Lemma \ref{lem:Demkey} and Theorem \ref{thm:orthog} thus ensure that
\[
(\omega_1+dd^c u)^n \leq  {\bf 1}_{\{u =v\}}(\omega_2+dd^c u)^n \leq {\bf 1}_{\{u =v\}}(\omega_2+dd^c v)^n. 
\]
We therefore  obtain $v_-(\omega_1) \leq v_-(\omega_2)$. This proves \eqref{eq:v-+}.

     \smallskip
     
    Let now $\omega,\omega'$ be two Hermitian metrics
    (we simplify notations).
    Observe  that $v_{\pm}(A \omega)=A^n v_{\pm}(\omega)$ for all $A>0$.
    Since $A^{-1} \omega' \leq \omega \leq A \omega$ for an appropriate choice of the constant $A>1$,
   items 1) and 2) follow from \eqref{eq:v-+}. 
 
    \end{proof}

   We now establish bounds on the mixed Monge-Amp\`ere quantities: 
    
    \begin{prop} \label{pro:easybounds}
    \text{ }
    \begin{enumerate}
    \item One always has $  v_{+,1}(\omega) <+\infty$.
    \item If $\omega$ is Hermitian then $0 < v_{-,1}(\omega)$.
    \item If $dd^c \omega^{n-2}=0$  then $v_{+,2}(\omega)<+\infty$.
    \item If $dd^c \omega=0$ and $dd^c \omega^2=0$ then 
      $v_{-,j}(\omega)=v_{+,j}(\omega)=V_{\omega} \in \R_+^*$.
     \item For all $0 \leq \ell \leq j \leq n$ one has $v_{+,\ell}(\omega) \leq 2^j v_{+,j}(\omega)$.     
    \item  $v_{+,n-1}(\omega)<+\infty$ if and only if $v_{+,n}(\omega)<+\infty$.
    \end{enumerate}
    \end{prop}
    
    A Hermitian metric  such that $dd^c (\omega^{n-2})=0$ is called Astheno-K\"ahler.
    These metrics play an important role in the study of harmonic maps
    (see \cite{JY93}).
        A Hermitian metric   satisfying    $dd^c \omega=0$
    is called SKT 
    or pluriclosed in the literature.
    When $n=3$ the Astheno-K\"ahler and the pluriclosed condition coincide,
    and the third item is due to Chiose \cite[Question 0.8]{Chi16}.
     Examples of compact complex manifolds admitting a pluriclosed
    metric can be found in \cite{FPS04,Ot20}.
    
    \smallskip
   
   Condition  (4)  has been introduced by Guan-Li  in \cite{GL10}. 
     It has been shown   by Chiose \cite{Chi16} that it is equivalent
     to the invariance of Monge-Amp\`ere masses:
    $\int_X (\omega+dd^c \varphi)^n=\int_X \omega^n$
    for all smooth $\omega$-psh functions if and only if $dd^c \omega^j=0$ for all $j=1,2$.
       Note that any compact complex surface admits a
    Gauduchon metric $dd^c \omega=0$ \cite{Gaud77}, which also satisfies
    $dd^c \omega^2=0$ for  bidegree reasons.
    
    \begin{proof}
   One can assume without loss of generality that $\omega \leq \tilde{\omega}$, where $\tilde{\omega}$ is a Gauduchon metric.
   It follows that for any $\f \in\PSH(X,\omega) \cap L^{\infty}(X)$,
   $$
   \int_X (\omega+dd^c \f) \wedge \omega^{n-1} \leq \int_X (\omega+dd^c \f) \wedge \tilde{\omega}^{n-1} 
   =\int_X  \omega \wedge \tilde{\omega}^{n-1},
   $$
   hence $v_{+,1}(\omega) \leq \int_X  \omega \wedge \tilde{\omega}^{n-1}<+\infty$.
   
   If $\omega$ is Hermitian one can similarly bound from below $\omega$ by a Gauduchon form
   and conclude that $v_{-,1}(\omega)>0$.

   \smallskip
   
     To prove (3) we fix $\f \in\PSH(X,\omega) \cap L^{\infty}(X)$ with $\sup_X \f =0$. We want to show that  
     $\int_X (\omega+dd^c \f)^{2} \wedge \omega^{n-2} \leq M$ is uniformly bounded from above. It suffices to bound 
     $\int_X (\omega_X +dd^c \f)^2\wedge \omega^2$ from above, where $\omega_X$ is a Hermitian metric on $X$. By approximation we can further assume that $\f$ is smooth. A direct application of Stokes theorem yields
      \begin{eqnarray*}
    \int_X (\omega_X+dd^c \f)^2 \wedge \omega^{n-2} &=& \int_X \omega_X^2\wedge \omega^{n-2}+
    2 \int_X \omega^{n-2} \wedge \omega_X  \wedge dd^c \f\\
    &+& \int_X \omega^{n-2} \wedge (dd^c \f)^2 \\
    &=&\int_X \omega_X^2\wedge \omega^{n-2}
    +  2 \int_X \f dd^c (\omega^{n-2}\wedge \omega_X) \\
    &-& \int_X \f dd^c \omega^{n-2} \wedge dd^c \f. 
    \end{eqnarray*}
   The latter integral vanishes since $dd^c \omega^{n-2}=0$. The second one is uniformly bounded
   since the functions $\f$ belong to a  compact subset of $L^1(X)$.
   Altogether this shows that $v_{+,2}(\omega)<+\infty$ if $dd^c (\omega^{n-2})=0$.
   
   \smallskip
   
   Assume now that  $dd^c \omega=0$ and $dd^c (\omega^2)=0$.
       Since $dd^c (\omega^2)=2 d \omega \wedge d^c \omega+2 \omega \wedge dd^c \omega$,
    this condition is equivalent to $dd^c \omega=0$ and $d \omega \wedge d^c \omega=0$ 
    which is also equivalent to $dd^c \omega^k =0$, for all $k=1,...,n$. 
    For $\f \in \mathcal{C}^{\infty}(X)$ the binomial expansion of 
    the Monge-Amp\`ere measure $(\omega+dd^c \f)^n$ yields
 \[
   \int_X (\omega+dd^c \f)^n = \sum_{k=0}^n \binom{n}{k} (dd^c \f)^k \wedge \omega^{n-k}. 
 \]
  Setting   $S=(dd^c \f)^p \wedge \omega^{n-p-1}$, and noting that $dd^c S=0$, we have   
  \[
 (dd^c \varphi )  \wedge S = (dd^c \varphi )  \wedge S  - \varphi dd^c S = d(d^c \f \wedge S -  \f \wedge d^c S).
  \]
  Stokes theorem thus yields $ \int_X (\omega+dd^c \f)^n=\int_X \omega^n$. 
If $\varphi\in \PSH(X,\omega)\cap L^{\infty}(X)$, we approximate $\f$ by a decreasing sequence of smooth $\omega_X$-psh functions 
$\varphi_j$. For each $j$ we have  $\int_X (\omega+dd^c \f_j)^n =\int_X \omega^n$. The signed measures $(\omega+dd^c \f_j)^n$   converge to $(\omega+dd^c \f)^n$, hence the   mass equality holds for $\f$, proving 4).

   \smallskip 
   
   Observe that for  any $\f \in\PSH(X,\omega) \cap L^{\infty}(X)$ and $0 \leq \ell \leq j \leq n$ one has
   \begin{equation} \label{eq:v+interm}
    \int_X (\omega+dd^c \f)^{\ell} \wedge \omega^{n-\ell} \leq \int_X (2\omega+dd^c \f)^j \wedge \omega^{n-j}
   \leq 2^j v_{+,j}(\omega).
   \end{equation}
    In particular $v_{+,n-1}(\omega) \leq 2^n v_{+,n}(\omega)$ hence
    $v_{+,n}(\omega)<+\infty \Rightarrow v_{+,n-1}(\omega)<+\infty$.
       We finally show conversely that $v_{+,n-1}(\omega)<+\infty  \Rightarrow v_{+,n}(\omega)<+\infty$
   by proving
   \begin{equation*}
   v_{+,n}(\omega) \leq  2^{2n-2} v_{+,n-1}(\omega).
   \end{equation*}
Observe indeed that 
  \begin{flalign*}
  0&= \int_X (\omega + dd^c \varphi -\omega)^n \\
  &= \int_X (\omega+ dd^c \varphi)^n + \sum_{k=1}^n (-1)^{k} \binom{n}{k}(\omega+dd^c \varphi)^{n-k} \wedge \omega^{k}\\
    	&\geq \int_X (\omega+ dd^c \varphi)^n - \sum_{1\leq 2k+1 \leq n} \binom{n}{2k+1} (\omega+dd^c \varphi)^{n-2k-1} \wedge \omega^{2k+1}. 
  \end{flalign*}
  Using \eqref{eq:v+interm} we thus get 
  \[
  v_{+,n}(\omega) \leq \sum_{1\leq 2k+1 \leq n} \binom{n}{2k+1} 2^{n-1} v_{+,n-1}(\omega)= 2^{2n-2} v_{+,n-1}(\omega). 
  \]

    \end{proof}

   \subsubsection{Uniformly bounded functions}

   Restricting to uniformly bounded $\omega$-psh functions, it is natural to consider  
     $$
    v^-_{M}(\omega):=\inf_{} \left\{ \int_X (\omega+dd^c u)^n  \; : \;  
    u \in\PSH(X,\omega)  \text{ with } -M \leq u \leq 0 \right\}
    $$
   where $M \in \R^+$,  and
    $$
   v^+_{M}(\omega):=\sup_{} \left\{ \int_X (\omega+dd^c u)^n  \; : \; 
    u \in\PSH(X,\omega)  \text{ with } -M \leq u \leq 0 \right\}.
    $$
    These quantities are always under control as we now explain:
   
   \begin{prop} \label{pro:bddMA2}
     Assume $\omega$ is non-collapsing.
  For any $M \in \R^+$,  one has
   $$
   0 <  v^-_{M}(\omega) \leq  v^+_{M}(\omega) <+\infty.
   $$
   \end{prop}

   \begin{proof}
The finiteness of $ v^+_{M}(\omega)$ follows easily from integration by parts,
it is e.g. a simple consequence of \cite[Theorem 3.5]{DK12}.

In order to show that $v^-_M(\omega)$ is positive we argue by contradiction.
Assume there exists $u_j \in\PSH(X,\omega)$ such that
$-M \leq u_j \leq 0$ and $\int_X (\omega+dd^c u_j)^n \leq 2^{-j}$.
For $j \in \N$ fixed, the sequence
$$
k \mapsto v_{j,k}:=P_{\omega}( \min(u_j,u_{j+1},\ldots,u_{j+k}))
$$
 decreases towards a $\omega$-psh function $w_j$ such that $-M \leq w_j \leq 0$.
 It follows therefore from Lemma \ref{lem:envmin}   that
 $$
 \int_X (\omega+dd^c w_{j})^n =\lim_{k \rightarrow +\infty} \int_X (\omega+dd^c v_{j,k})^n 
 \leq \sum_{\ell=0}^{+\infty} \int_X (\omega+dd^c u_{j+\ell})^n
 \leq 2^{-j+1}.
 $$
Thus the sequence $j \mapsto w_j$ increases to a bounded $\omega$-psh function $w$
such that $(\omega+dd^c w)^n=0$, which yields a contradiction.
   \end{proof}

     \begin{exa} \label{exa:collapsing}
We provide here an example of a  semi-positive form $\omega$ such that $\int_X \omega^n>0$
but $\omega$ is collapsing, in particular $v_-(\omega)=0$.  
Let $X=Y\times Z$ where $Y,Z$ are two compact complex manifolds of dimension $m\geq 1$, $p\geq 1$ respectively, and ${\rm dim}X=n=p+m$. Fix a smooth function $u$ on $Y$ such that $\omega_Y+dd^c u <0$ is negative in a small open set $U\subset Y$. Let $0\leq \rho\leq 1$ be a cut-off function on $Y$  supported in $U$. The smooth $(1,1)$-form $\omega$ defined by  
$$
\omega = \rho \circ \pi_1 (\pi_1^* \omega_Y + \pi_2^* \omega_Z).
$$
is semipositive on $X$ and satisfies $\omega(y,z)=0$ for $y\notin U$. 

Set now $\phi:= P_{\omega}(u\circ \pi_1)$ and let  $\mathcal{C}:= \{\phi=u\circ \pi_1\}$ denote the contact set. 
The Monge-Amp\`ere measure $(\omega+dd^c \phi)^n$ is concentrated on $\mathcal{C}$. 
Arguing as in \cite[Proposition 3.1]{Ber07} one can show that 
$\mathcal{C} \subset \{ x \in X, \; \omega+dd^c u \circ \pi_1 (x) \geq 0 \}$.
Since $\omega+dd^c (u\circ \pi_1)<0$ is negative in $U\times Z$, it follows that $\mathcal{C}\subset X\setminus (U\times Z)$. 
Now $\omega=0$ outside $U\times Z$, hence 
\[
(\omega+dd^c \phi)^n \leq {\bf 1}_{\mathcal{C}} (dd^c  u\circ \pi_1)^n =0,
\] 
because $u\circ \pi_1$ depends only on $y$. It thus follows that $(\omega+dd^c \phi)^n =0$ on $X$. 
 \end{exa}

    \subsection{Bimeromophic invariance}

\begin{lem} \label{lem:onedirection}
Let $f:X \rightarrow Y$ be a proper holomorphic   map
    between  compact complex manifolds of dimension $n$,
    equipped with Hermitian forms $\omega_X,\omega_Y$.
    Then
    \begin{itemize}
    \item  $v_+(\omega_X)<+\infty \Longrightarrow v_+(\omega_Y)<+\infty$;
    \item  $v_-(\omega_Y)>0 \Longrightarrow v_-(\omega_X)>0$ if $f$ has connected fibers.
    \end{itemize}
\end{lem}

It follows from Zariski's main theorem that $f$ has connected fibers if it is bimeromorphic.

\begin{proof}
Up to rescaling, we can assume   that $f^* \omega_Y \leq \omega_X$.
Fix $\f \in\PSH(Y,\omega_Y) \cap L^{\infty}(Y)$. Then $\f \circ f \in\PSH(X,\omega_X) \cap L^{\infty}(X)$ with
$$
\int_Y (\omega_Y+dd^c \f)^n =\int_X (f^* \omega_Y+dd^c \f \circ f)^n 
\leq \int_X (\omega_X+dd^c \f \circ f)^n \leq v_+(\omega_X),
$$
thus $v_+(\omega_Y) \leq v_+(\omega_X)$ and the first assertion is proved.

\smallskip

Consider now $\p \in\PSH(X,\omega_X) \cap L^{\infty}(X)$  and set $u=P_{f^*\omega_Y}(\psi)$.
The function $u$ is $f^*\omega_Y$,  hence plurisubharmonic on the fibers of $f$.
If the latter are  connected we obtain that
 $u$ is constant on them, i.e. $u=\f \circ f$ for some function $\f \in\PSH(Y,\omega_Y) \cap L^{\infty}(Y)$.
Since $(f^* \omega_Y +dd^c u)^n\leq {\bf 1}_{\{ u=\p\}} (f^* \omega_Y +dd^c \p)^n$, we infer
$$
v_-(\omega_Y) \leq \int_Y (\omega_Y+dd^c \f)^n =\int_X (f^* \omega_Y +dd^c u)^n \leq \int_X (\omega_X +dd^c \p)^n
$$
so that $v_-(\omega_Y) \leq v_-(\omega_X)$, proving the second assertion.
\end{proof}

    We conversely show that   the properties $v_+(\omega_X)<+\infty$ and $v_-(\omega_X)>0$
    are  invariant under blow ups and blow downs with smooth centers:

    \begin{thm} \label{thm:bimerom}
    Let $X$ and $Y$  be compact complex manifolds which are bimeromorphically equivalent.
    Then
    \begin{itemize}
    \item  $v_+(\omega_X)<+\infty$  if and only if $v_+(\omega_Y)<+\infty$;
    \item  $v_-(\omega_X)>0$  if and only if $v_-(\omega_Y)>0$.
    \end{itemize}
    \end{thm}
    
    \begin{proof}
       A celebrated result of Hironaka (see e.g. \cite[Theorem 5.1]{Mat}) ensures that  any bimeromorphic map between compact complex manifolds
   is a finite composition of blow ups and blow downs with smooth centers. We can thus assume that 
   $f: X\rightarrow Y$  is the blow up of $Y$ along a smooth center.

   We fix $\p$ a quasi-plurisubharmonic function  such that
   $\pi^* \omega_Y+dd^c \p \geq \delta \omega_X$. The existence of $\psi$ follows from a classical argument in complex geometry (see  \cite{BL70}, \cite[Proposition 3.2]{FT09}). By Demailly's approximation theorem we can further assume that $\psi$ has analytic singularities. Up to scaling  we can assume without loss of generality that $\delta=1$,
   and we set $\Omega=\{x \in X \; : \;  \p(x)>-\infty\}$.
   
   We already know by Lemma \ref{lem:onedirection} that 
     $v_+(\omega_X)<+\infty \Longrightarrow v_+(\omega_Y)<+\infty$. 
     Assume conversely that  $v_+(\omega_Y)<+\infty$. 
   For any $\f \in\PSH(X,\omega_X) \cap L^{\infty}(X)$,
  \begin{eqnarray*}
   \int_X (\omega_X+dd^c \f)^n &\leq&  \int_{\Omega} (\pi^* \omega_Y+dd^c (\p+\f))^n. 
  \end{eqnarray*}
  The function $u=\p+\f$ is $\pi^*\omega_Y$-psh and bounded from above in $\Omega$.
  It is constant on the fibers of $\pi$, hence 
  $u=v \circ \pi$ with $v \in\PSH(\pi(\Omega),\omega_Y) \cap L^{\infty}(\pi(\Omega))$.
  As $v$ is bounded from above, it extends trivially through the analytic set $\pi(\partial \Omega)$ as a 
  $\omega_Y$-psh function which is locally bounded in $\pi(\Omega)$. Thus
   \begin{flalign*}
    \int_{\Omega} (\pi^* \omega_Y+dd^c (\p+\f))^n &= \int_{\pi(\Omega)} (\omega_Y+dd^c v)^n\\
    &\leq \liminf_{j\to +\infty}\int_{\pi(\Omega)} (\omega_Y+dd^c \max(v,-j))^n \\
    & \leq v_+(\omega_Y)
   \end{flalign*}
   yields $v_+(\omega_X) \leq v_+(\omega_Y)<+\infty$.

   \smallskip
   
   We now assume that $v_-(\omega_X)>0$.
   Pick $v \in\PSH(Y,\omega_Y) \cap L^{\infty}(Y)$ and set $u=P_{\omega_X}(v \circ \pi -\p)$.
   Observe that $u \in\PSH(X,\omega_X) \cap L^{\infty}(X)$
   and recall that $(\omega_X+dd^c u)^n$ is concentrated on the contact set 
   ${\mathcal C}=\{u+\p=v \circ \pi\}$ (see Theorem \ref{thm:orthog}).
   Observe that $\pi^*\omega_Y +dd^c (u+\psi) \geq \omega_X+dd^c u \geq 0$. Since $v \circ \pi$ is also $\pi^*\omega_Y$-
   psh, locally bounded in $\Omega$,
   with $u+\p \leq v \circ \pi$, it follows from 
   Lemma \ref{lem:Demkey}  that
   $$
   1_{\mathcal C} (\pi^* \omega_Y+dd^c (u+\p))^n \leq  1_{\mathcal C} (\pi^* \omega_Y+dd^c v \circ \pi)^n
   \leq (\pi^* \omega_Y+dd^c v \circ \pi)^n.
   $$
   Now $\pi^* \omega_Y+dd^c (u+\p) \geq \omega_X+dd^c u$
   and $(\omega_X+dd^c u)^n$ is concentrated on ${\mathcal C}$ so
   $$
    1_{\mathcal C} (\pi^* \omega_Y+dd^c (u+\p))^n \geq (\omega_X+dd^c u)^n.
   $$
   We infer
   \begin{eqnarray*}
   v_-(\omega_X) \leq \int_X (\omega_X+dd^c u)^n
   &\leq&  \int_{\mathcal C} (\pi^* \omega_Y+dd^c (u+\p))^n \\
   &\leq & \int_X (\pi^* \omega_Y+dd^c v \circ \pi)^n
   =\int_Y (\omega_Y+dd^c v)^n,
   \end{eqnarray*}
   showing that $v_-(\omega_Y) \geq v_-(\omega_X)>0$.
        The reverse implication  $v_-(\omega_Y)>0 \Longrightarrow v_-(\omega_X)>0$  follows from Lemma \ref{lem:onedirection}.
    \end{proof}

     Recall that  a compact complex manifold $X$ belongs to the Fujiki class ${\mathcal C}$
   if there exists a holomorphic bimeromophic map $\pi:Y \rightarrow X$,
   where $Y$ is compact K\"ahler.
   Since $v_+(\omega_X)=v_-(\omega_X)=\int_X \omega_X^n \in \R_+^*$
   when $\omega_X$ is a K\"ahler form, we obtain the following:
   
        \begin{coro} \label{cor:Fujiki}
      If $X$ belongs to the Fujiki class $\mathcal{C}$ then 
      $$
      0<v_-(\omega_X) \leq v_+(\omega_X)<+\infty.
      $$ 
      \end{coro}

    \section{Weak transcendental   Morse inequalities}

      \subsection{Nef and big forms}

    Recall that the Bott-Chern cohomology group $H_{BC}^{1,1}(X,\R)$
    is the quotient of   closed real smooth $(1,1)$-forms,
   by the image of  ${\mathcal C}^{\infty}(X,\R)$
   under the $dd^c$-operator.
    This is a finite dimensional vector space as $X$ is compact.
    
    Nefness and bigness  are fundamental positivity properties of holomorphic line bundles
    in complex algebraic geometry (see  \cite{Laz}).
    Their transcendental counterparts have been defined and studied by Demailly (see \cite{Dem12}):
    
    \begin{defi}
    \text{ }
    
    \begin{itemize}
    \item A cohomology class $\alpha \in H_{BC}^{1,1}(X,\R)$ is nef if 
    for any $\e>0$, one can find a smooth closed real $(1,1)$-form
    $\theta_{\e} \in \alpha$ such that $\theta_{\e} \geq -\e \omega_X$.
    \item   A {\it Hermitian current} on $X$ is a positive current $T$ of bidegree $(1,1)$
   which dominates a Hermitian form, i.e. there exists $\delta>0$ such that
   $T \geq \delta \omega_X$.
   \item  A cohomology class $\alpha \in H_{BC}^{1,1}(X,\R)$ is big if 
  it can be represented by a closed Hermitian current
  (a K\"ahler current).
    \end{itemize}
     \end{defi}

    It follows from an approximation  result of Demailly \cite{Dem92}
    that one can weakly approximate a Hermitian current by
    Hermitian currents with analytic singularities. In particular
    a big cohomology class can be represented by
    a K\"ahler current with analytic singularities.

      \smallskip
    
         By analogy with the above setting, we propose the following definitions:
    
     \begin{definition}
     Let $\omega$ be a smooth real $(1,1)$ form on $X$.
     \begin{itemize}
      \item We  say that   $\omega$ is nef if for any $\e>0$ there exists a smooth  quasi-plurisubharmonic function $\f_{\e}$ such that 
      $\omega+dd^c \f_{\e} \geq -\e \omega_X$.
     \item   We  say that   $\omega$ is big if there exists a $\omega$-psh function $\rho$ with analytic singularities such that $\omega+dd^c \rho$ dominates a Hermitian form.
     \end{itemize}   
    \end{definition}
    
  Note that $\PSH(X,\omega)$ is non empty in both cases: indeed $\rho \in \PSH(X,\omega)$ in the latter case,
while  one can extract $\f_{\e_j} \rightarrow \f \in \PSH(X,\omega)$  in the former, normalizing the
potentials $\f_{\e_j}$ by imposing $\sup_X \f_{\e_j}=0$.

 When $X$ is a compact K\"ahler manifold and 
      $\alpha \in H^{1,1}_{BC}(X,\R)$ is   nef with $\alpha^n>0$,
      a celebrated result of Demailly-P\u{a}un \cite[Theorem 0.5]{DP04} 
      ensures the existence of a K\"ahler current representing $\alpha$.
      This result is the key step in establishing a transcendental
      Nakai-Moishezon criterion (see \cite[Main theorem]{DP04}).
    
   We study in the sequel  a possible extension of this result to the Hermitian setting.
   We thus need to extend the definition of $v_-$ to nef forms:
   
   \begin{defi}
   If $\omega$ is a nef $(1,1)$-form, we set
   $$
   \hat{v}_-(\omega):=\inf_{\e>0} v_-(\omega+\e \omega_X).
   $$
   \end{defi}
   
   Although the form $\omega+\e \omega_X$ needs not be semi-positive, one can find
   by definition  
    a semi-positive form $\omega+\e \omega_X+dd^c \f_{\e}$
   cohomologous to $\omega+\e \omega_X$,
   and it is understood here that $v_-(\omega+\e \omega_X):=v_-(\omega+\e \omega_X+dd^c \f_{\e})$.
  By \eqref{eq:v-+}, the definition of $\hat{v}_-(\omega)$ is independent of the choice of the Hermitian form $\omega_X$.

   It is natural to expect that 
   this definition is consistent with the previous one when $\omega$ is semi-positive,
   and that $\hat{v}_-(\omega)=\alpha^n$ when $\omega$ is a closed form representing
   a nef class $\alpha \in H^{1,1}_{BC}(X,\R)$:
   
   \begin{lem}\label{lem: lem v hat and v}
If $\omega$ is semi-positive then $\hat{v}_-(\omega)=v_-(\omega)$.  If  $v_{+}(\omega_X) <+\infty$ and $\omega$ is a closed form representing a nef class in $H^{1,1}_{BC}(X,\R)$, then  $\hat{v}_-(\omega)=\alpha^n$. 
   \end{lem}
   
   When $X$ is K\"ahler, it is classical that any nef class $\alpha \in H_{BC}^{1,1}(X,\R)$ satisfies
$\alpha^n \geq 0$. This inequality is no longer obvious on an arbitrary Hermitian manifold
(we thank J.-P.Demailly for emphasizing this issue)
but, as a consequence of the above lemma, it remains true when $v_+(\omega_X)<+\infty$.

   \begin{proof}
   Assume first that $\omega$ is semi-positive and set $\omega_{\varepsilon}:= \omega+\varepsilon \omega_X$, for $\varepsilon\in (0,1)$. Proposition \ref{pro:bddsMA} ensures that  $v_-(\omega) \leq v_-(\omega_{\varepsilon})$, hence $v_-(\omega) \leq \hat{v}_-(\omega)$. On the other hand, for any $u\in \PSH(X,\omega)\cap L^{\infty}(X)$ we have 
   \begin{flalign*}
   	\int_X (\omega+dd^c u)^n &=\int_X (\omega_{\varepsilon} +dd^c u-\varepsilon \omega_X)^n \\
   	&\geq \int_X (\omega_{\varepsilon}+dd^c u)^n  -C\varepsilon\\
   	& \geq \hat{v}_-(\omega) -C\varepsilon,
   \end{flalign*}
   where $C$ is a constant depending on $u$, but it is harmless as we will let $\varepsilon \to 0$ while keeping $u$ fixed. Doing so we obtain $\int_X (\omega+dd^c u)^n \geq \hat{v}_-(\omega)$, and taking infimum over such $u$ we obtain $v_-(\omega) \geq \hat{v}_-(\omega)$, proving the first statement.   
    
    \smallskip
   
   Assume now that $\omega$ is closed and $\{\omega\} \in H^{1,1}_{BC}(X,\R)$ is nef. We can also assume that $-\omega_X \leq \omega\leq \omega_X$.  We pick $\f \in \PSH(X,\omega+\e \omega_X) \cap C^{\infty}(X)$
  and observe that $\PSH(X,\omega+\e \omega_X) \subset \PSH(X,2 \omega_X)$ for $0<\e \leq 1$, hence
 $$
   \int_X (\omega+\e \omega_X +dd^c \f)^n
   = \int_X (\omega+dd^c \f)^n+\sum_{j=1}^n \left( \begin{array}{c}
   n \\ j 
   \end{array} \right)
   \e^j \int_X \omega_X^j \wedge (\omega+dd^c \f)^{n-j}.
   $$
   Writing $\omega+dd^c \f=(2\omega_X+dd^c \f)-(2\omega_X-\omega)$, expanding
   $(\omega+dd^c \f)^{n-j}$ accordingly and using $0 \leq 2\omega_X -\omega \leq 3 \omega_X$,
   we obtain that $ \left | \int_X \omega_X^j \wedge (\omega+dd^c \f)^{n-j} \right |$ is bounded from above by a 
   finite sum of terms   $\int_X \omega_X^{\ell} \wedge (\omega_X+dd^c \f)^{n-\ell}$,
   each of which is bounded from above by $3^n v_+(\omega_X)$. Since $\int_X (\omega+dd^c \f)^n=\alpha^n$,
   we end up with
   $$
   \alpha^n - C \e v_+(\omega_X)\leq \int_X (\omega+\e \omega_X +dd^c \f)^n
   \leq \alpha^n+C \e v_+(\omega_X),
   $$
   using that $\e^j \leq \e$ for all $1 \leq j \leq n$. We infer $\hat{v}_-(\omega)=\alpha^n$.
   \end{proof}
  
     \subsection{Demailly-P\u{a}un conjecture}
     
          \subsubsection{Hermitian currents}
     
      The following is a natural generalization of \cite[Conjecture 0.8]{DP04}:

  \begin{quest} \label{quest:dempaun}
  Let $X$ be a compact complex manifold.
  Let $\omega$ be a  nef
  $(1,1)$-form  such that     $\hat{v}_-(\omega)>0$.
  Does there exist a $\omega$-psh function $\f$ with analytic singularities such that
   the   current   $\omega+dd^c \f $ dominates a Hermitian form ?
  \end{quest}

   We provide a partial answer to Question \ref{quest:dempaun}
   following some ideas of Chiose \cite{Chi13}:
      
      \begin{theorem} \label{thm:dempaun} 
      Let $\omega$ be a nef $(1,1)$-form. 
      \begin{itemize}
      \item If $\hat{v}_{-}(\omega)>0$ and  $v_+(\omega_X)<+\infty$ then $\omega$ is big. 
      \item Conversely if $\omega$ is big and $v_-(\omega_X)>0$ then $\hat{v}_-(\omega)>0$.
      \end{itemize}
       \end{theorem}
      \begin{proof}
    We assume without loss of generality that $\omega \leq \omega_X/2$.
   
   We first assume that $\hat{v}_{-}(\omega)>0$, $v_+(\omega_X)<+\infty$, and we prove that $\omega$ is big. An application of Hahn-Banach theorem as in \cite[Lemma 3.3]{Lam99} shows that
    the existence of a Hermitian current $\omega+dd^c \psi \geq \delta \omega_X$ is equivalent to
    the inequalities
    $$
    \int_X \omega \wedge \theta^{n-1} \geq \delta \int_X \omega_X \wedge \theta^{n-1},
    $$
    for all Gauduchon metrics $\theta$. 
     Assume by contradiction that there exists a sequence
    of Gauduchon metrics $\theta_j$ such that
    $$
    \int_X \omega \wedge \theta_j^{n-1} \leq \frac{1}{j} \int_X \omega_X \wedge \theta_j^{n-1}.
    $$
    We can normalize the latter so that $\int_X \omega_X \wedge \theta_j^{n-1}=1$.
    
    Set $\omega_j=\omega+\frac{1}{j}\omega_X$ and note that $\omega_j \leq \omega_X$ for $j \geq 2$.
   Since $\omega$ is nef, one can find $\p_j \in {\mathcal C}^{\infty}(X,\R)$
   such that $\omega_j+dd^c \p_j$ is a Hermitian form, hence
   the main result of  \cite{TW10} ensures that there exist constants
   $C_j>0$ and
   $\f_j \in\PSH(X,\omega_j) \cap {\mathcal C}^{\infty}(X)$ such that
   $\sup_X \f_j=0$ and
   $$
   (\omega_j+dd^c \f_j)^n=C_j \omega_X \wedge \theta_j^{n-1}.
   $$
   It follows from Proposition \ref{pro:bddsMA} that 
   $$
   C_j=\int_X (\omega_j+dd^c \f_j)^n \geq v_-(\omega_j) \geq \hat{v}_-(\omega)>0,
   $$
  while by assumption
  $
  \int_X (\omega_j+dd^c \f_j)^{n-1} \wedge \omega_X \leq M:=v_{+,n-1}(\omega_X)
  $
   is   bounded from above.

   We set  $\alpha_j:=\omega_j+dd^c \f_j$ and consider
   $$
   E:=\{ x \in X , \; \omega_X \wedge \alpha_j^{n-1} \geq 2M \omega_X \wedge \theta_j^{n-1} \}.
   $$
   This set has small $\omega_X \wedge \theta_j^{n-1}$ measure since
   $$
   \int_E \omega_X \wedge \theta_j^{n-1} \leq \frac{1}{2M} \int_E \omega_X \wedge \alpha_j^{n-1} \leq \frac{1}{2},
   $$
   thus $\int_{X \setminus E} \omega_X \wedge \theta_j^{n-1} \geq \frac{1}{2}$,
   thanks to the normalization $\int_X \omega_X \wedge \theta_j^{n-1}=1$.
   
   We can compare $\omega_X$ and $\alpha_j$ in $X \setminus E$ since
   $$
   \omega_X \wedge \alpha_j^{n-1} \leq 2M \omega_X \wedge \theta_j^{n-1}=\frac{2M}{C_j} \alpha_j^n
   \leq \frac{2M}{\hat{v}_-(\omega)} \alpha_j^n.
   $$
   Thus $\alpha_j \geq \frac{\hat{v}_-(\omega)}{2nM} \omega_X$ in $X \setminus E$ and we infer
   $$
   \int_{X \setminus E} \alpha_j \wedge \theta_j^{n-1} \geq 
 \frac{\hat{v}_-(\omega)}{2nM}  \int_{X \setminus E} \omega_X \wedge \theta_j^{n-1} \geq \frac{\hat{v}_-(\omega)}{4nM} >0,
   $$
   which contradicts
 \begin{eqnarray*}
    \int_{X} \alpha_j \wedge \theta_j^{n-1} 
    &=&\int_X \omega \wedge \theta_j^{n-1}+\frac{1}{j} \int_X \omega_X \wedge \theta_j^{n-1}+\int_X dd^c \f_j \wedge \theta_j^{n-1} \\
    &\leq& \frac{2}{j} \int_X \omega_X \wedge \theta_j^{n-1}= \frac{2}{j} \rightarrow 0,
 \end{eqnarray*}
  where   $\int_X dd^c \f_j \wedge \theta_j^{n-1}=0$ follows from the Gauduchon property of $\theta_j$.

   We next assume that $\omega$ is big, $v_{-}(\omega_X)>0$, and we prove that $\hat{v}_{-}(\omega)>0$  by an argument similar to that of Theorem \ref{thm:bimerom}.  Fix a $\omega$-psh function $\psi$ with analytic singularities such that $\omega+dd^c \psi \geq \delta \omega_X$ for some $\delta >0$. We can assume that $\delta =1$ and $\sup_X \psi =0$. We prove that $v_{-}(\omega+\varepsilon \omega_X) \geq v_{-}(\omega_X)$ for all $\varepsilon>0$.  Fix $\varepsilon>0$, $u\in \PSH(X,\omega+\varepsilon \omega_X)\cap L^{\infty}(X)$, and set $v=P_{\omega_X}(u-\psi)$. The open set $G=\{\psi >-1\}$ is not empty hence it is non-pluripolar.  On $G$ we have $u \leq u-\psi \leq u+1 \leq \sup_X u+1$. We then get $v-\sup_X u -1\leq V_{G,\omega}^*$, where $V_{G,\omega}$ is the extremal function of $E$ (see the proof of Lemma \ref{lem: Josefson semipositive} for its definition). It follows that $v$ is a bounded $\omega_X$-psh function and $(\omega_X+dd^c v)^n$ is supported on the contact set $\mathcal{C}= \{v= u-\psi\} \subset \{\psi>-\infty\}$. Since $v+\psi \leq u$ with equality on $\{\psi>-\infty\} \cap \mathcal{C}= \mathcal{C}$, Lemma \ref{lem:Demkey} ensures that
		\[
		{\bf 1}_{\{\psi>-\infty\} \cap \mathcal{C}}(\omega +\varepsilon \omega_X + dd^c (v+\psi))^n \leq {\bf 1}_{\{\psi>-\infty\} \cap \mathcal{C}}(\omega +\varepsilon \omega_X +dd^c u)^n. 
		\]
		Using $\omega+dd^c \psi \geq \omega_X$ and the fact that $(\omega_X+dd^c v)^n(\psi=-\infty)=0$ since $v$ is bounded, we thus arrive at (noting that $(\omega_X+dd^c v)^n$ is supported on $\mathcal{C}$)
		\[
		\int_X (\omega_X+dd^c v)^n = \int_{\mathcal{C}} (\omega_X+dd^c v)^n \leq \int_X (\omega+\varepsilon \omega_X+ dd^c u)^n.
		\]
		We thus get  $v_{-}(\omega+\varepsilon \omega_X) \geq v_{-}(\omega_X)>0$, for all $\varepsilon>0$, hence $\hat{v}_{-}(\omega) >0$. 
		    \end{proof}

              This result provides in particular the following   answer to Question \ref{quest:dempaun}:

          \begin{coro}    \label{cor:dempaun}
          The answer to Question \ref{quest:dempaun} is positive if 
   \begin{itemize}
   \item either $n=2$  ($X$ is any compact surface);
   \item or $n=3$ and  $X$ admits a pluriclosed metric;
   \item or $n$ is arbitrary and $X$ belongs to the Fujiki class;
    \item orelse $n$ is arbitrary and $X$ admits  a Guan-Li metric.
\end{itemize}          
      \end{coro}
      
      Let us stress that the $2$-dimensional setting is due to 
      Buchdahl \cite{Buch99} and Lamari \cite{Lam99}.
      The three dimensional case follows from Proposition \ref{pro:easybounds}.

      \subsubsection{Transcendental Grauert-Riemenschneider conjecture}
      
      Let $L \rightarrow X$ be a semi-positive holomorphic line bundle with $c_1(L)^n>0$.
      An influential conjecture of Grauert-Riemenschneider \cite{GR70} asked whether 
      the existence of such a line bundle implies that $X$ is Moishezon
      (i.e. bimeromorphically equivalent to a projective manifold).
      
      This conjecture has been solved positively by Siu in \cite{Siu84} (see also \cite{Dem85}).
      Demailly and P\u{a}un have proposed a transcendental version of this conjecture:
      
      \begin{conj} \cite[Conjecture 0.8]{DP04}
      Let $X$ be a compact complex manifold of dimension $n$.
      Assume that $X$ posseses a nef class $\alpha \in H_{BC}^{1,1}(X,\R)$ such that
      $\alpha^n>0$. Then      $X$ belongs to the Fujiki class.
      \end{conj}
      
          As a direct consequence of Theorem \ref{thm:dempaun}, Lemma \ref{lem: lem v hat and v}, and Corollary \ref{cor:Fujiki},
we obtain the following answer to the transcendental  Grauert-Riemenschneider conjecture:

    \begin{thm}    \label{thm:GrauRiem}
    Let $X$ be a compact $n$-dimensional complex manifold.
    Let $\alpha \in H_{BC}^{1,1}(X,\R)$ be a nef class such that $\alpha^n>0$.
    	  The following are equivalent:
    	  \begin{itemize}
    	  	\item   $\alpha$ contains a K\"ahler current
    	  	\item  $v_+(\omega_X) <+\infty$.  			
    	  \end{itemize}
\end{thm}

 Since a K\"ahler current with analytic singularities can be desingularized
 after finitely many blow-ups producing a K\"ahler form, we obtain:

\begin{coro}
Let $\alpha \in H_{BC}^{1,1}(X,\R)$ be a nef class such that $\alpha^n>0$.
Then  $X$ belongs to the Fujiki class
if and only if $v_+(\omega_X)<+\infty$.
\end{coro}

 \subsection{Transcendental holomorphic Morse inequalities}
 
The following conjecture has been proposed by Boucksom-Demailly-P\u{a}un-Peternell,
as a transcendental counterpart to the
holomorphic Morse inequalities for integral classes due to Demailly:

\begin{conj}\cite[Conjecture 10.1.ii]{BDPP13}
 Let $X$ be a compact $n$-dimensional complex manifold.
Let $\alpha, \beta \in H^{1,1}_{BC}(X,\mathbb{C})$ be nef classes such that 
 	$
 	\alpha^n > n \alpha^{n-1} \cdot  \beta. 
 	$
 	
 	Then $\alpha -\beta$ contains a K\"ahler current and ${\rm Vol}(\alpha-\beta) \geq \alpha^n - n \alpha^{n-1} \cdot  \beta$.
\end{conj}

Note that this contains  \cite[Conjecture 0.8]{DP04} as a particular case ($\beta=0$).
This conjecture has recently been established by Witt Nystr\"om \cite{WN19} when $X$ is projective.
Building on works of Xiao \cite{Xiao15} and Popovici \cite{Pop16}
we propose the following characterization which answers the qualitative part:
 
 \begin{theorem} \label{thm:morse}
 	Let $\alpha, \beta \in H^{1,1}_{BC}(X,\mathbb{C})$ be nef classes such that 
 	$
 	\alpha^n > n \alpha^{n-1} \cdot  \beta. 
 	$
 	The following are equivalent:
 	\begin{itemize}
 	\item  $\alpha -\beta$ contains a K\"ahler current;
 	\item $v_+(\omega_X)<+\infty$.
 	\end{itemize}
 \end{theorem}
 
 \begin{proof}
 If $\alpha -\beta$ contains a K\"ahler current, then $X$ belongs to the Fujiki class and we have already observed 
 that $v_+(\omega_X)<+\infty$ (see Corollary \ref{cor:Fujiki}).
 
 \smallskip
 
 We now assume that $v_+(\omega_X)<+\infty$.
 	Let $\omega$ and $\omega'$ be smooth closed real $(1,1)$-forms representing $\alpha$ and $\beta$ respectively. We can assume 
 	without los of generality that $\omega \leq \frac{\omega_X}{2}$ and $\omega' \leq \frac{\omega_X}{2}$.  
 	For each $\varepsilon>0$ we fix smooth functions 
 	$\varphi_{\varepsilon} \in \PSH(X,\omega+ \varepsilon \omega_X)$ and $\psi_{\varepsilon} \in \PSH(X,\omega'+ \varepsilon \omega_X)$
 	such that $\omega_{\varepsilon}:= \omega+ \varepsilon\omega_X + dd^c \varphi_{\varepsilon}$ and $\omega_{\varepsilon}'= \omega'+ \varepsilon \omega_X +dd^c \psi_{\varepsilon}$ are hermitian forms.
 	
 	Assume by contradiction that $\alpha-\beta$ does not contain any K\"ahler current. 
 	It follows from Hahn-Banach theorem as in \cite[Lemma 3.3]{Lam99}   that there exist Gauduchon metrics $\eta_{\varepsilon}$ 
 	such that 
 	\begin{equation}\label{eq: Morse 1}
 		\int_X (\omega_{\varepsilon} -\omega_{\varepsilon}') \wedge \eta_{\varepsilon}^{n-1} 
 		\leq \varepsilon 	\int_X \omega_{\varepsilon}' \wedge \eta_{\varepsilon}^{n-1}. 
 	\end{equation}
 	We normalize  $\eta_{\varepsilon}$ so that $\int_X \omega_{\varepsilon}' \wedge \eta_{\varepsilon}^{n-1}=1$.
 	
 	Using \cite{TW10} we can find  unique constants $c_{\e}>0$ and normalized  functions
 	$u_{\varepsilon} \in \PSH(X,\omega_{\varepsilon}) \cap {\mathcal C}^{\infty}(X)$
 	such that 
 	\[
 	(\omega_{\varepsilon}+ dd^c u_{\varepsilon})^n = c_{\varepsilon} \omega_{\varepsilon}' \wedge \eta_{\varepsilon}^{n-1}, \; \sup_X u_{\varepsilon}=0.
 	\] 
 	Our normalization for $\eta_{\e}$ yields 
 	$c_{\varepsilon}  = \int_X (\omega_{\varepsilon}+ dd^c u_{\varepsilon})^n$. 
Applying Lemma \ref{lem: CS Popovici} below with $\theta_1= \omega_{\varepsilon}+ dd^c u_{\varepsilon}$, 
 	$\theta_2= c_{\varepsilon}\omega_{\varepsilon}'$ and $\theta_3= \eta_{\varepsilon}$, and recalling that 
 	$\theta_1^n= \theta_2 \wedge \theta_3^{n-1}$ with 
 	$\int_X  \theta_1^{n}= \int_X \theta_2 \wedge \theta_3^{n-1}=c_{\e}$, 
 	 we obtain 
 	\[
 	\left (\int_X (\omega_{\varepsilon}+dd^c u_{\e}) \wedge \eta_{\varepsilon}^{n-1} \right ) \left ( \int_X (\omega_{\varepsilon}+dd^c u_{\varepsilon})^{n-1} \wedge \omega_{\varepsilon}' \right ) \geq  \frac{c_{\varepsilon}}{n}.
 	\]
 	Now 
 	$\int_X (\omega_{\varepsilon}+dd^c u_{\e}) \wedge \eta_{\varepsilon}^{n-1} =\int_X \omega_{\varepsilon}  \wedge \eta_{\varepsilon}^{n-1} $
 	because $\eta_{\varepsilon}$ is a Gauduchon metric, 
 	while \eqref{eq: Morse 1} yields 
 	$\int_X \omega_{\varepsilon}  \wedge \eta_{\varepsilon}^{n-1} \leq (1+\e) \int_X \omega'_{\varepsilon}  \wedge \eta_{\varepsilon}^{n-1}=(1+\e)$, hence
 	\[
 	(1+\varepsilon) \int_X (\omega_{\varepsilon}+dd^c u_{\varepsilon})^{n-1} \wedge \omega_{\varepsilon}' \geq \frac{1}{n}\int_X (\omega_{\varepsilon} + dd^c u_{\varepsilon})^n. 
 	\]
 		We finally claim that, as  $\varepsilon \to 0$,
 	\[
 	\int_X (\omega_{\varepsilon} + dd^c u_{\varepsilon})^n \to \alpha^n\; \text{and}\; 
 	\int_X (\omega_{\varepsilon}+dd^c u_{\varepsilon})^{n-1} \wedge \omega_{\varepsilon}' \to \alpha^{n-1} \cdot \beta,
 	\]
 which yields the contradiction	 $n \alpha^{n-1} \cdot \beta \geq \alpha^n$.
 
 \smallskip
 
 We first explain why $\int_X (\omega_{\varepsilon} + dd^c u_{\varepsilon})^n \to \alpha^n$.   Stokes theorem yields
 \begin{eqnarray*}
 \lefteqn{
 \alpha^n =\int_X (\omega+dd^c (u_{\e}+\f_{\e}))^n=\int_X (\omega+\e \omega_X +dd^c (u_{\e}+\f_{\e})-\e \omega_X)^n }\\
 &=& \int_X (\omega_{\e}+dd^c u_{\e})^n+\sum_{j=0}^{n-1} 
 \left( \begin{array}{c} n \\ j  \end{array} \right) \e^{n-j} (-1)^{n-j} \int_X (\omega_{\e} +dd^c u_{\e})^j \wedge \omega_X^{n-j}.
 \end{eqnarray*}
 Since $\omega \leq \frac{\omega_X}{2}$, the function $v_{\e}=u_{\e}+\f_{\e}$ is $\omega_X$-psh for
 $0<\e \leq \frac{1}{2}$, hence
 $$
 0 \leq \int_X (\omega_{\e} +dd^c u_{\e})^j \wedge \omega_X^{n-j} \leq 
 \int_X (\omega_X +dd^c v_{\e})^j \wedge \omega_X^{n-j} \leq 2^n v_+(\omega_X),
 $$
 as follows from Proposition \ref{pro:easybounds}. We infer
 $$
 \left| \alpha^n -\int_X (\omega_{\e}+dd^c u_{\e})^n \right|
 \leq \sum_{j=0}^{n-1} 
 \left( \begin{array}{c} n \\ j  \end{array} \right) \e^{n-j}  
 2^n v_+(\omega_X)
 \leq 4^n \e \, v_+(\omega_X).
 $$
 The conclusion thus follows by letting $\e \to 0$.
 
 We similarly can check that 
 $$
\left|  \alpha^{n-1} \cdot \beta-\int_X (\omega_{\varepsilon}+dd^c u_{\varepsilon})^{n-1} \wedge \omega_{\varepsilon}' \right| 
\leq 2 \cdot 6^n \e \, v_+(\omega_X).
 $$
 Using Stokes theorem again we indeed obtain that  
 	\begin{flalign*}
 		\alpha^{n-1} \cdot \beta &
 		=\int_X (\omega +dd^c \varphi_{\varepsilon} + dd^c u_{\varepsilon})^{n-1}\wedge (\omega'+dd^c \psi_{\varepsilon})\\
 		&= \int_X (\omega_{\e} +  dd^c u_{\varepsilon}   -\varepsilon \omega_X)^{n-1}\wedge (\omega_{\e}'  -\varepsilon \omega_X)\\
 		&= \int_X (\omega_{\varepsilon}+dd^c u_{\varepsilon})^{n-1}\wedge \omega'_{\varepsilon} + O(\e). 
 	\end{flalign*}
Each term $\int_X (\omega_X+dd^c v_{\e})^{\ell} \wedge (\omega_X+dd^c \p_{\e})^p \wedge \omega_X^q$,
with $\ell+p+q=n$, is bounded from above by $3^n v_+(\omega_X)$, as one can check by observing that
the function $\frac{v_{\e}+\p_{\e}}{3}$ is $\omega_X$-psh with 
$$
\int_X (\omega_X+dd^c v_{\e})^{\ell} \wedge (\omega_X+dd^c \p_{\e})^p \wedge \omega_X^q 
\leq 3^n \int_X  \left(\omega_X+dd^c \frac{v_{\e}+\p_{\e}}{3} \right)^n.
$$
 \end{proof}

We have used in the previous proof the following observation of Popovici:

 \begin{lemma}\label{lem: CS Popovici}
 Let $\theta_1,\theta_2,\theta_3$ be hermitian forms  on $X$. Then
 	\[
 	\left ( \int_X \theta_1 \wedge \theta_3^{n-1} \right ) \left ( \int_X \theta_1^{n-1} \wedge \theta_2 \right )
 	\geq \frac{1}{n} \left(\int_X   \sqrt{\frac{\theta_2 \wedge \theta_3^{n-1}}{\theta_1^n}}\theta_1^n \right)^2.
 	\]
 	In particular if   $\theta_1^n=\theta_2 \wedge \theta_3^{n-1}$, then
 		\[
 	\left ( \int_X \theta_1 \wedge \theta_3^{n-1} \right ) \left ( \int_X \theta_1^{n-1} \wedge \theta_2 \right )
 	\geq \frac{1}{n} \left(\int_X   \theta_1^n \right)^2.
 	\]
 \end{lemma}
 
 We provide the proof as a courtesy to the reader.
 
 \begin{proof}
 It follows from Cauchy-Schwarz inequality that 
 $$
 \left ( \int_X \theta_1 \wedge \theta_3^{n-1} \right ) \left ( \int_X \theta_1^{n-1} \wedge \theta_2 \right )
 	\geq \left ( \int_X \sqrt{\frac{\theta_1 \wedge \theta_3^{n-1}}{\theta_1^n} \frac{\theta_1^{n-1} \wedge \theta_2}{\theta_1^n}} \theta_1^n  \right )^2.
 	  $$
 The elementary pointwise estimate
 	\[
 	Tr_{\theta_3}(\theta_1) Tr_{\theta_1}(\theta_2) \geq Tr_{\theta_3}(\theta_2). 
 	\]
is \cite[Lemma 3.1]{Pop16}. Multiplying by $\frac{\theta_3^n}{\theta_1^n}$ it can be reformulated as 
\begin{equation} \label{eq:pop}
\frac{\theta_1 \wedge \theta_3^{n-1}}{\theta_1^n} 
\cdot \frac{\theta_2 \wedge \theta_1^{n-1}}{\theta_1^n} 
\geq \frac{1}{n} \frac{\theta_2 \wedge \theta_3^{n-1}}{\theta_1^n}.
\end{equation}
The first inequality follows. Moreover 
when $\theta_1^n=\theta_2 \wedge \theta_3^{n-1}$, we infer
$$
\int_X   \sqrt{\frac{\theta_2 \wedge \theta_3^{n-1}}{\theta_1^n}}\theta_1^n 
=  \int_X \theta_1^n.
$$
 \end{proof}

  Motivated by possible extensions of the conjectures of Demailly-P\u{a}un and
 Boucksom-Demailly-P\u{a}un-Peternell, we introduce the following:

  \begin{defi}
  Given $\omega_1,\ldots,\omega_n$ hermitian forms we  consider
  $$
  v_-(\omega_1,\ldots,\omega_n):=\inf \left\{ \int_X (\omega_1+dd^c \f_1) \wedge \cdots \wedge (\omega_n+dd^c \f_n), \;
  \f_j \in {\mathcal P}(\omega_j) \right\},
  $$
  and 
   $$
  v_+(\omega_1,\ldots,\omega_n):=\sup \left\{ \int_X (\omega_1+dd^c \f_1) \wedge \cdots \wedge (\omega_n+dd^c \f_n), \;
  \f_j \in {\mathcal P}(\omega_j) \right\},
  $$
  where ${\mathcal P}(\omega_j):=\PSH(X,\omega_j) \cap L^{\infty}(X)$.
  If the $\omega_j$'s are merely nef we set
  $$
   \hat{ v}_-(\omega_1,\ldots,\omega_n):=\inf_{\e>0}   v_-(\omega_1+\e \omega_X ,\ldots,\omega_n+\e \omega_X).
  $$
  and
    $$
   \hat{ v}_+(\omega_1,\ldots,\omega_n):=\inf_{\e>0}   v_+(\omega_1+\e \omega_X ,\ldots,\omega_n+\e \omega_X).
  $$
  \end{defi}
  
  A straightforward generalization of Theorem \ref{thm:morse} along the lines of Theorem \ref{thm:dempaun}
  is the following:
  
  \begin{thm} \label{thm:morse2}
   Let $X$ be a compact $n$-dimensional complex manifold such that $v_+(\omega_X)<+\infty$.
  Let $\omega,\omega'$ be nef $(1,1)$-forms.
  If $ \hat{ v}_-(\omega)>n  \hat{ v}_+(\omega,\ldots,\omega, \omega')$ then 
  the form $\omega-\omega'$ is big.
  \end{thm}
  
We leave the technical details to the reader.

\end{document}